\documentclass[oneside,12pt]{amsart}
\usepackage{amssymb, mathrsfs, amscd}
\usepackage[all]{xy}

\usepackage{graphics}
\usepackage{graphicx}
\usepackage{enumerate}

\newtheorem*{thma}{Theorem A}
\newtheorem*{thmb}{Theorem B}

\newtheorem{cor}{Corollary}
\newtheorem{theorem}{Theorem}[section]
\newtheorem{definition}[theorem]{Definition}
\newtheorem{lemma}[theorem]{Lemma}

\newtheorem{proposition}[theorem]{Proposition}
\newtheorem{corollary}[theorem]{Corollary}
\newtheorem*{corollary*}{Corollary}

\newtheorem{remark}[theorem]{Remark}

\author{Dieter Degrijse}
\address{Department of Mathematics, KU Leuven, Kortrijk, Belgium}%
\email{Dieter.Degrijse@kuleuven-kortrijk.be}%

\author{Nansen Petrosyan}
\address{Department of Mathematics, KU Leuven, Kortrijk, Belgium}%
\email{Nansen.Petrosyan@kuleuven-kortrijk.be}%
\thanks{Both authors were supported by the Research Fund KU Leuven.}
\thanks{The second author was also supported by the FWO-Flanders Research Fellowship.}

\title[]{Bredon cohomological dimensions for groups acting on CAT($0$)-spaces}
\date{\today}
\newcommand{\mF}{\mathcal {F}}
\newcommand{\mv}{\underline{\mathrm{vst}}}
\newcommand{\ms}{\underline{\mathrm{st}}}
\newcommand{\mms}{\underline{\underline{\mathrm{st}}}}
\newcommand{\Z}{\mathbb Z}

\newcommand{\orb}{\mathcal{O}_{\mF}G}
\newcommand{\orbmod}{\mbox{Mod-}\mathcal{O}_{\mF}G}

\newcommand{\nathom}{\mathrm{Hom}_{\mF}}

\bigskip
\begin{document}
\maketitle
\begin{abstract} Let $G$ be a group acting isometrically with discrete orbits on a separable complete CAT(0)-space of bounded topological dimension.
Under certain conditions, we give upper bounds for the Bredon cohomological dimension of $G$ for the families of finite and virtually cyclic subgroups. As an application, we prove that the  mapping class group of any closed, connected, and orientable surface of genus $g\geq 2$ admits a $(9g-8)$-dimensional classifying space  with  virtually cyclic stabilizers. In addition, our results apply  to fundamental groups of graphs of groups and groups acting on Euclidean buildings. In particular, we show that all finitely generated linear groups of positive characteristic have a finite dimensional  classifying space for proper actions and a finite dimensional  classifying space for the family of virtually cyclic subgroups. We also show that every generalized Baumslag-Solitar group has a 3-dimensional model for the classifying space with virtually cyclic stabilizers.
\end{abstract}
\section{Introduction}
Let $G$ be a discrete group and let $\mathcal{F}$ be a family of subgroups of $G$, i.e.~a collection of subgroups of $G$ that is closed under conjugation and taking subgroups. A \emph{classifying space of $G$ for the family $\mathcal{F}$} is a $G$-CW-complex $X$  such that $X^H$ is contractible for every $H$ in $\mathcal{F}$ and  empty when  $H$ is not in  $\mathcal{F}$ (see \cite{tom}). Equivalently, one can characterize $X$ by the property that for any $G$-CW-complex $Y$ with  stabilizers in $\mF$, there exists, up to $G$-homotopy,  a unique $G$-map from $Y$ to $X$. Motivated by the  the Baum-Connes and Farrell-Jones Isomorphism Conjectures, there is a particular interest to study classifying spaces for the families of finite and virtually cyclic subgroups. These conjectures predict isomorphisms between certain equivariant cohomology theories of classifying spaces of $G$ and $K$- and $L$-theories of reduced group $C^*$-algebras and of group rings of $G$  (see e.g. \cite{BCH}, \cite{FJ}, \cite{MV},\cite{LuckReich}). Other applications of classifying spaces for the family of finite subgroups include computations in group cohomology and the formulation of a generalization from finite to infinite groups of the Atiyah-Segal Completion Theorem in topological K-theory  (see \cite[\S7-8]{Luck2}).
With these applications in mind, it is always desirable to have models for $E_{\mathcal{F}}G$ with good geometric properties. One such property is the dimension of $E_{\mathcal{F}}G$. Although a classifying space always exists for any discrete group and a family of subgroups, it need not be finite dimensional. The smallest possible dimension of a model for $E_{\mathcal{F}}G$  is an invariant of the group called the \emph{geometric dimension of $G$ for the family $\mathcal{F}$} and denoted by $\mathrm{gd}_{\mathcal{F}}G$.  

In the present article our aim is  to study  the geometric dimension of  groups that act isometrically on separable CAT(0)-spaces of finite topological dimension.   We do not require the action to be  proper but only to have discrete orbits. Our results extend the main theorem   of \cite{Luck3} from proper actions to actions with discrete orbits, and from proper to separable CAT(0)-spaces.
This allows us to consider examples of  groups that admit actions with infinite stabilizer subgroups on complete, not necessarily proper CAT(0)-spaces, such as finitely generated linear groups of positive characteristic and mapping class groups.

In this approach we make use of  Bredon cohomology which  allows one to analyze finiteness properties of $E_{\mathcal{F}}$G  using homological techniques. For instance, given a discrete group $G$ and a family of subgroups $\mathcal{F}$, there is a notion of Bredon cohomological dimension $\mathrm{cd}_{\mathcal{F}}G$ which satisfies the inequality 
\[ \mathrm{cd}_{\mathcal{F}}G \leq  \mathrm{gd}_{\mathcal{F}}G \leq \max\{3, \mathrm{cd}_{\mathcal{F}}G \}. \] Thus, to show that there exists a finite dimensional model for $E_{\mathcal{F}}G$, it suffices to prove that the Bredon cohomological dimension of $G$ for the family $\mathcal{F}$ is finite. We recall the definition and the necessary properties of Bredon cohomology in Section 3.

In order  to apply  Bredon cohomology in our context,  in Section 2, we associate to the  isometric action of a group on a metric space  a certain cellular action. This is done in Proposition \ref{prop: G-map}, which may be of separate interest to the reader. It asserts that if a group $G$ acts isometrically and with discrete orbits on a separable metric space $X$ then
there exists a $G$-CW-complex $Y$ of dimension at most the topological dimension of $X$ for which the stabilizers are subgroups of point stabilizers of $X$ together with a $G$-map $f: X \rightarrow Y$.

Before stating our main results, let us establish some notation and terminology. 
A group $G$ will always be assumed to be  discrete. If $\mathcal{F}$ is the family of  finite or the family of virtually cyclic subgroups of a groups $G$, then $\mathrm{cd}_{\mathcal{F}}G$ will be denoted by $\underline{\mathrm{cd}}G$ or by $\underline{\underline{\mathrm{cd}}}(G)$, respectively. Suppose $G$ acts on a topological space $X$. We say that $G$ acts \emph{discretely} on $X$ if the orbits $G\cdot x$ are discrete subsets of $X$, for all $x \in X$. Let $\mathcal{E}(G,X)$ be the set containing all groups $E$ that fit into a short exact sequence
$1 \rightarrow N \rightarrow E \rightarrow F \rightarrow 1$, 
where $N$ is a subgroup of the stabilizer $G_x$ for some point $x \in X$ and $F$ is a subgroup of a finite dihedral group. Finally, we define the following values associated to the pair $(G,X)$, which will be used throughout the article
\begin{eqnarray*}
\mv(G,X) &=&\sup\{  \underline{\mathrm{cd}}(E) \;\ | \ E \in  \mathcal{E}(G,X)\} \\
\ms(G,X) &=& \sup\{ {\underline{\mathrm{cd}}}(G_x) \ | \ x \in X  \}  \\
\mms(G,X) &=& \sup\{ \underline{\underline{\mathrm{cd}}}(G_x) \ | \ x \in X  \}.
\end{eqnarray*}

\smallskip

\indent  Clearly, one has $$\ms(G,X) \leq \mv(G,X), \;\;\; \mbox { and } \;\;\; \ms(G,X) \leq \mms(G,X)+1$$ since it is known that $\underline{\mathrm{cd}}(S) \leq\underline{\underline{\mathrm{cd}}}(S)+1$ for any group $S$ (e.g. see \cite[4.2]{DemPetTal}). We are not aware of an example of a group $G$ acting on a space $X$ such that $\ms(G,X)$ is finite but $\mms(G,X)$ or $\mv(G,X)$ are not. On the other hand, using Theorem C of  \cite{DP2} as in \cite[6.5]{DP2}, one can construct examples where $G$ is an integral linear group and $X$ is a point such that both $\mms(G,X)$ and  $\mv(G,X)$ are  arbitrarily larger than    $\ms(G,X)$. \\

In Section $4$, we prove the following.
\begin{thma}  Let $G$ be a group acting isometrically and discretely on a separable CAT(0)-space $X$ of topological dimension $n$,  and 
let $\mathcal{F}$ be a family of subgroups of $G$ such that $X^{H} \neq \emptyset$ for all $H \in \mathcal{F}$. Suppose that there exists an integer $d \geq 0$ such that for each $x \in X$ one has $\mathrm{cd}_{\mathcal{F}\cap G_x}(G_x)\leq d$. Then we have \[\mathrm{cd}_{\mathcal{F}}(G)\leq d+n.\]
\end{thma}
The following corollary is immediate.
\begin{cor}\label{cor: stab} Let $G$ be a group acting isometrically and discretely on a separable CAT(0)-space $X$ of topological dimension $n$, and 
let $\mathcal{F}$ be the smallest family of subgroups of $G$ containing the point stabilizers $G_x$, for every $x \in X$.  Then we have \[\mathrm{cd}_{\mathcal{F}}(G)\leq n.\]
\end{cor}
Since each isometric action of a finite group on complete CAT($0$)-space has a global fixed point (see Corollary II.2.8(1) in \cite{BridHaef}), we conclude the following from Theorem A.
\begin{cor}\label{cor: finite} Let $G$ be a group acting isometrically and discretely on a complete separable CAT(0)-space $X$ of topological dimension $n$. Then 
\[  \underline{\mathrm{cd}}(G)\leq \ms(G,X)+n. \]
\end{cor}
The next theorem provides an upper bound for $\underline{\underline{\mathrm{cd}}}(G)$.
\begin{thmb} \label{th: b}
Let $G$ be a countable group acting discretely by semi-simple isometries on a complete separable CAT(0)-space $X$ of topological dimension $n$. Then 
\[\underline{\underline{\mathrm{cd}}}(G)\leq \max\{\mms(G, X), \mv(G, X)+1\}+n. \]
\end{thmb}
Let us note that the assumption of the theorem that $G$ acts by semi-simple isometries is satisfied when $X$ is a proper metric space on which $G$ acts cocompactly (see Proposition \ref{prop: cocompact}), or when $X$ is an $\mathbb{R}$-tree (see \cite[II.6.6(3)]{BridHaef}), or when $X$ is a piecewise Euclidean complex with finite shapes on which $G$ acts by cellular isometries (see \cite[Theorem A]{Bridson}). 

In the last section, Section 5, we consider several concrete applications of the above theorems. We discuss these next.
\begin{cor} \label{cor: intro specific stab}Let $G$ be a countable group acting discretely by semi-simple isometries on a complete separable CAT(0)-space $X$ of topological dimension $n$. \begin{itemize}
\item[(i)] If $G_x$
is finite for each $x \in X$, then we have
\begin{equation*}
\underline{\mathrm{cd}}(G)  \leq  n \ \;\;\; \mbox{and} \;\;\; \  \  \underline{\underline{\mathrm{cd}}}(G)  \leq n+1.
\end{equation*}
\item[(ii)] If $G_x$
is virtually free for each $x \in X$, then we have
\begin{equation*}
\underline{\mathrm{cd}}(G)  \leq  n+1 \ \;\;\; \mbox{and} \;\;\; \  \  \underline{\underline{\mathrm{cd}}}(G)  \leq n+2.
\end{equation*}
\item[(iii)]If $G_x$
is virtually polycyclic of Hirsch length at most $h$ for each $x \in X$, then 
\begin{equation*}
\underline{\mathrm{cd}}(G)  \leq  n+h   \ \;\;\; \mbox{and} \;\;\; \  \ \underline{\underline{\mathrm{cd}}}(G)  \leq n+h+1.
\end{equation*}
\item[(iv)]If $G_x$
is elementary amenable of Hirsch length at most $h$ for each $x \in X$, then 
\begin{equation*}
\underline{\mathrm{cd}}(G)  \leq  n+h+1  \  \;\;\; \mbox{and} \;\;\; \  \ \underline{\underline{\mathrm{cd}}}(G)  \leq n+h+2.
\end{equation*}
\end{itemize}
\end{cor}

 Since every simplicial tree can be viewed as a one-dimensional CAT(0)-space, this result applies to fundamental groups of  graphs of groups (see \cite{Serre}) and in particular to generalized Baumslag-Solitar groups. By definition, a generalized Baumslag-Solitar group $G$ is  a fundamental group of a graph of groups where all vertex and edge groups are infinite cyclic. In this case, we can actually determine the Bredon cohomological dimension of $G$.

\begin{cor} \label{cor: intro baumslag}  Let $G$ be a generalized Baumslag-Solitar group, then
\[
\underline{\underline{\mathrm{cd}}}(G) = \left\{
\begin{array}{@{}ll@{\ }l}
	3 &  \text{if } &\Z^2\subseteq G,\\
	
	0 &  \text{if } &\Z \cong G,\\
	
	2 & \multicolumn{2}{l}{\text{otherwise}.}
\end{array} \right.
\]
\end{cor}

Another source of examples to which we can apply Theorems A and B are finitely generated linear groups of positive characteristic. By the fundamental work of Bruhat and Tits (see \cite{BT}),  such groups admit fixed-point-free actions on Euclidean buildings. These buildings  have a natural piecewise Euclidean metric  which turns out to be CAT(0).

\begin{cor} \label{cor: intro linear} Let $G$ be a finitely generated subgroup of $\mathrm{GL}_n(F)$ where  $F$ is a field of positive characteristic. Then $${\underline{\mathrm{cd}}}(G)<\infty \;\;\; \mbox{ and } \;\;\; \underline{\underline{\mathrm{cd}}}(G)<\infty.$$
\end{cor}

Lastly, we present an application to the mapping class group  of  any closed, connected, and orientable surface $S_g$ of genus $g\geq 2$. This group  acts   by semi-simple isometries on the completion of the  Teichm\"{u}ller space ${\mathcal{T}}(S_g)$ equipped the Weil-Petersson metric, such that the stabilizer subgroups are finitely generated virtually abelian groups of Hirsch length at most $3g-3$. It follows that the action is also discrete. Since ${\mathcal{T}}(S_g)$ is a separable CAT(0)-space, we obtain the following result.

\begin{cor} \label{cor: intro mapping class group} Let $S_g$ be a closed, connected and orientable surface of genus $g\geq 2$, and let $\mathrm{Mod}(S_g)$ be its mapping class group. Then we have
\[   \underline{\underline{\mathrm{cd}}}(\mathrm{Mod}(S_g)) \leq 9g-8.            \]
\end{cor}

\section{Discrete isometric group actions}
Throughout this section, let $X$ be a  metric space and let $G$ be a discrete group acting on $X$ by isometries. For every   $\varepsilon \geq 0$ and each $x \in X$, we denote by $\overline{B(x, \varepsilon)}$ the closure of the open ball $\mathrm{B}(x,\varepsilon)$ with radius $\varepsilon$ centered at $x$.  
The action of an element $g \in G$ on a point $x \in X$ will be denoted by $g \cdot x$ and the associated orbit space  by $G\setminus X$. The action of $G$ on $X$ is called \emph{cocompact} if there exists a compact subset $K$ of $X$ such that $X = \bigcup_{g \in G}g\cdot K$.
\begin{definition} \rm We say that $G$ acts \emph{discretely} on $X$ if for every $x \in X$, the orbit $G\cdot x$ is a discrete subset of $X$.
\end{definition}
The following lemma gives some equivalent definitions of a discrete action.
\begin{lemma}\label{lemma: discrete} The following are equivalent.
\begin{itemize}
\item[(i)] The group $G$ acts discretely on $X$.
\smallskip
\item[(ii)] For every $x \in X$, there exists an $\varepsilon > 0$ such that for all $g \in G$ 
\[ g \cdot \mathrm{B}(x,\varepsilon)\cap \mathrm{B}(x,\varepsilon) \neq \emptyset \Leftrightarrow g \in G_{x}. \]
\smallskip
\item[(iii)]  For every $x \in X$, there exists an $\varepsilon> 0$ such that for every subset $K$ of $X$ that can be covered by finitely many open balls with radius $\varepsilon$,
there exist elements $g_1,\ldots,g_n \in G$ such that
\[  S:=\{ g \in G \ | \ K \cap g \cdot \mathrm{B}(x,\varepsilon)\neq \emptyset \}\subseteq\bigcup_{i=1}^n g_iG_x.\]
\end{itemize}
\end{lemma}
\begin{proof}
Let us begin by showing that (i) implies (ii). Let $x \in X$. Since $G\cdot x$ is a discrete subset of $X$, there exists an $\varepsilon > 0$ such that $\mathrm{B}(x,2\varepsilon) \cap G\cdot x =\{x\}$. This implies that $g \in G_x$, whenever we have $g \cdot x \in \mathrm{B}(x,2\varepsilon)$. Now assume that $g \cdot \mathrm{B}(x,\varepsilon)\cap \mathrm{B}(x,\varepsilon) \neq \emptyset$ for some $g \in G$. This can only be the case if $d(x,g \cdot x)< 2\varepsilon$. Hence $g \cdot x \in \mathrm{B}(x,2\varepsilon)$ and therefore $g \in G_x$. Conversely, if $g \in G_x$ then one obviously has $g \cdot \mathrm{B}(x,\varepsilon)\cap \mathrm{B}(x,\varepsilon) \neq \emptyset$. This shows that (i) implies (ii).

Next, let us  prove that (ii) implies (iii). Let $x$ be a point of $X$. By (ii), there is an $\varepsilon > 0$ such that for all $g \in G$ 
\[ g \cdot \mathrm{B}(x,2\varepsilon)\cap \mathrm{B}(x,2\varepsilon) \neq \emptyset \Leftrightarrow g \in G_{x}. \]
Let $K$ be a subset of $X$ for which there exist elements $x_1,\ldots,x_m \in  X$ such that
\[ K\subseteq \bigcup_{i=1}^m \mathrm{B}(x_i,\varepsilon).\]
We may assume that $S$ is non-empty, otherwise there is nothing to prove. For each $i \in \{1,\ldots,m\}$ choose, when possible, an element $g_i \in S$ such that $g_i \cdot \mathrm{B}(x,\varepsilon)\cap \mathrm{B}(x_i,\varepsilon) \neq \emptyset$. Note that $g_iG_x \subseteq S$. Now let $g \in S$. Then $g \cdot \mathrm{B}(x,\varepsilon)\cap \mathrm{B}(x_i,\varepsilon) \neq \emptyset$ for some $i$. This implies that there is an  (already chosen) element $g_i \in S$ satisfying  $g_i \cdot \mathrm{B}(x,\varepsilon)\cap \mathrm{B}(x_i,\varepsilon) \neq \emptyset$. It follows that 
\[   \mathrm{B}(g\cdot x,2\varepsilon)\cap \mathrm{B}(g_i \cdot x,2\varepsilon) \neq \emptyset,         \]
and hence $g_i^{-1}g \cdot \mathrm{B}(x,2\varepsilon)\cap \mathrm{B}(x,2\varepsilon) \neq \emptyset$. We conclude that $g \in g_iG_x$. This proves that (ii) implies (iii).

Finally, we will argue that (iii) implies (i). Let $x \in X$.  By (iii), there exists an $\varepsilon>0$ and elements $g_1,\ldots,g_n \in G$ such that
\[  S=\{ g \in G \ | \  \mathrm{B}(x,\varepsilon)\cap g \cdot \mathrm{B}(x,\varepsilon)\neq \emptyset \}\subseteq\bigcup_{i=1}^n g_iG_x.\]
Note that if $S=G_x$ then (i) follows immediately. So, suppose $S$ contains an element that does not fix $x$. We define \[\delta=\frac{1}{2}\min\{  d(g_i \cdot x,x) \ | \  g_i \in G \smallsetminus G_x  \}\] 
and note that $\delta< \varepsilon$. We now claim that $\mathrm{B}(x,\delta)\cap G \cdot x = \{x\}$, Indeed, suppose by a way of contradiction that there exists  $g \in G \smallsetminus G_x$ such that $g\cdot x \in \mathrm{B}(x,\delta)$. Then this implies that $g$ is contained in \[\{ g \in G\smallsetminus G_x \ | \  \mathrm{B}(x,\varepsilon)\cap g \cdot \mathrm{B}(x,\varepsilon)\neq \emptyset \}\] and hence $g \in g_iG_x$ for some $g_i \in G \smallsetminus G_x$. Therefore, we have $g_i\cdot x \in \mathrm{B}(x,\delta)$. But then \[d(g_i\cdot x,x)< \delta = \frac{1}{2}\min\{  d(g_i \cdot x,x) \ | \  i \in \{1,\ldots,n \}  \},\] yielding a contraction. This proves the claim and the lemma. 
\end{proof}

\begin{remark} \rm \label{remark: discrete action}
\begin{enumerate} \item[]
\item It is clear from the proof of Lemma \ref{lemma: discrete} that, if for a given $x \in X$ (ii) is valid for $2\varepsilon >0$ then  (iii) is valid for that  $x$ and the value $\varepsilon$.
\smallskip
\item An isometric group action on a  metric space is proper (in the sense of \cite[I.8.2]{BridHaef}) if and only if it is discrete and has finite point stabilizers.
\end{enumerate}
\end{remark}
The following lemma is a generalization of  Proposition II.6.10(4) in \cite{BridHaef}.
\begin{lemma} \label{lemma: product} Suppose that $G$ acts discretely on the product of metric spaces $X\times Y$ via isometries in $\mathrm{Iso}(X)\times \mathrm{Iso}(Y)$. Let $N$ be a normal subgroup of $G$ consisting of elements that act identically on $X$ under the projection of $\mathrm{Iso}(X)\times \mathrm{Iso}(Y)$ onto $\mathrm{Iso}(X)$. Assume furthermore that $N$ acts cocompactly on $Y$ under the projection $\mathrm{Iso}(X)\times \mathrm{Iso}(Y)$ onto $\mathrm{Iso}(Y)$. Then $G/N$ acts discretely on $X$ under the projection of $\mathrm{Iso}(X)\times \mathrm{Iso}(Y)$ onto $\mathrm{Iso}(X)$, where $gN$ acts as $g$ for all $g \in G$.
\end{lemma}
\begin{proof} We may assume that $G$ is a subgroup of $\mathrm{Iso}(X)\times \mathrm{Iso}(Y)$ and denote an element of $G$ as $(g,\alpha)$ with $g \in \mathrm{Iso}(X)$ and $\alpha\in \mathrm{Iso}(y)$. Elements of $N$ are of the form $(\mathrm{Id}_X, a)$. By cocompactness, we can find a compact subset $K$ of $Y$ such that $Y=\bigcup_{(\mathrm{Id}_X,a)\in N}a\cdot K$. 
Let $(x,y) \in X \times Y$ and let $\varepsilon>0$ be chosen such that $2\varepsilon$ satisfies property (iii) of Lemma \ref{lemma: discrete} for the element $(x,y)\in X \times Y$. Now choose $\mu>0$ and $\delta>0$ small enough such that every subset of $X \times Y$ of the form $B(x,\mu) \times B(z,\delta)$, for $z\in Y$, is contained in some open ball of radius $2\varepsilon$. Because $K$ is compact it can be covered by finitely many open balls of radius $\delta$. Since $(x,y)$ and $2\varepsilon$ satisfy property (iii) of Lemma \ref{lemma: discrete},
there exist elements $(g_1,\alpha_1)\ldots (g_n,\alpha_n) \in G$ such that
\[S:=\Big\{  (g,\alpha) \in G \ | \   \Big(\mathrm{B}(x,\mu)\times K \Big)\cap (g,\alpha)\cdot \mathrm{B}((x,y),2\varepsilon)\neq \emptyset \Big\}  \subseteq \bigcup_{i=1}^n (g_i,\alpha_i) G_{(x,y)}   .         \]
Now, choose $\varepsilon_0 > 0$ such that $\varepsilon_0 \leq \varepsilon$ and $2\varepsilon_0 < d(g_i \cdot x,x) $ for all $i \in \{1,\ldots,n\}$ for which $g_i \cdot x \neq x$.  
Since $Y=\bigcup_{(\mathrm{Id}_X,a)\in N}a\cdot K$, any element of $G/N$ can be represented as $(g,\alpha)N$ such that $\alpha \cdot y \in K$. So, let $(g,\alpha)N \in G/N$  such that $\alpha\cdot y \in K$ and assume that 
\begin{equation}\label{eq: condition} g \cdot \mathrm{B}(x,\varepsilon_0)\cap \mathrm{B}(x,\varepsilon_0) \neq \emptyset. \end{equation}
Since $\varepsilon_0 \leq \varepsilon$, one can easily verify that $$(x, \alpha\cdot y)\in \Big(\mathrm{B}(x,\mu)\times K \Big)\cap (g,\alpha)\cdot \mathrm{B}((x,y),2\varepsilon)$$ and hence $(g,\alpha)\in S$. Thus, there exists  $j \in \{1,\ldots,n\}$ such that $(g,\alpha) \in (g_j,\alpha_j)G_{(x,y)}$. Since by (\ref{eq: condition}) we have $d(g\cdot x,x)< 2\varepsilon_0$, and $2\varepsilon_0 < d(g_i \cdot x,x) $ for all $i\in \{1,\ldots,n\}$ for which $g_i \cdot x \neq x$, we conclude that $g_j\in G_x$. This implies that $(g,\alpha)N \in (G/N)_x$ and completes the proof. 
\end{proof}
Let $S$ be a topological space and let $\mathcal{U}=\{U_{\alpha}\}_{\alpha \in I}$ be an open cover of $S$. The \emph{dimension} $\dim(\mathcal{U}) \in \mathbb{N}\cup \{\infty \}$ of $\mathcal{U}$ is the infimum over all integers $d\geq 0$ such that any finite collection of pairwise distinct elements $U_0,\ldots,U_{d+1}$ of $\mathcal{U}$ has the property that $\cap_{i=0}^{d+1}U_i = \emptyset$. A \emph{refinement} $\mathcal{V}=\{V_{\beta}\}_{\beta \in J}$ of $\mathcal{U}$ is an open cover of $S$ such that for every $V \in \mathcal{V}$ there is an $U \in \mathcal{U}$ with $V \subseteq U$. The \emph{topological dimension} $\dim(S) \in \mathbb{N}\cup \{\infty \}$ of the space $S$ is the infimum over all integers $d\geq 0$ such that any open cover $\mathcal{U}$ of $S$ has a refinement $\mathcal{V}$ with $\dim(\mathcal{V})\leq d$. 

There are several other notions of dimension that can be associated to a topological space $S$. One can for example only consider finite open covers  $\mathcal{U}=\{U_{i}\}_{i=1,\ldots,k}$ of $S$ and their finite refinements $\mathcal{V}=\{V_{j}\}_{j=1,\ldots,r}$ of $\mathcal{U}$. The number $\mathrm{dim}_F(S)$  is then defined to be the infimum over all integers $d\geq 0$ such that any open finite cover $\mathcal{U}$ of $S$ has a finite refinement $\mathcal{V}$ with $\dim(\mathcal{V})\leq d$.  By relaxing finite covers to locally finite covers, one obtains the invariant $\mathrm{dim}_{LF}(S)$. Recall that an open cover $\mathcal{U}$ is locally finite if every point $x \in X$ has an open neighbourhood $V$ such that $V$ intersects only finitely many opens of $\mathcal{U}$. Finally, one can also define the small inductive dimension $\mathrm{ind}(S)$ of $S$ and the large inductive dimension $\mathrm{Ind}(S)$ of $S$. We refer the reader to \cite{Engelking} for the definitions of these invariants. 

For a general topological space, these various notions of dimension can differ. However, in the case of separable metric spaces, it turns out that they all coincide. 
\begin{lemma}\label{lemma: dim} Let $X$ be a paracompact Hausdorff space. Then, we have
\[    \mathrm{dim}_{F}(X) =\mathrm{dim}_{LF}(X)=\mathrm{dim}(X). \]
If, in addition, $X$ is a separable metric space, then
\[   \mathrm{dim}_{F}(X) =\mathrm{dim}_{LF}(X)=\mathrm{dim}(X)=\mathrm{ind}(X)=\mathrm{Ind}(X).  \]

\end{lemma}
\begin{proof} Since paracompact Hausdorff spaces  are normal, it follows from Theorem 3.5 in \cite{dowker} that $\mathrm{dim}_{F}(X) =\mathrm{dim}_{LF}(X)$. Next, let $\mathcal{U}$ be an open cover of $X$. Since $X$ is paracompact, $\mathcal{U}$ has a locally finite refinement $\mathcal{U}_0$. Now,  $\mathcal{U}_0$ and hence $\mathcal U$ has a locally finite refinement $\mathcal{V}$ satisfying $\dim(\mathcal{V})\leq \mathrm{dim}_{LF}(X)$. We conclude that $ \mathrm{dim}(X) \leq  \mathrm{dim}_{LF}(X) $. 

 Next, let $\mathcal{U}=\{U_{i}\}_{i=1,\ldots,k}$ be a finite open cover of $X$. Then $\mathcal{U}$ has a (possibly infinite) refinement $\mathcal{V}=\{V_{\alpha}\}_{\alpha \in I}$ satisfying $\dim(\mathcal{V})\leq \dim(X)$. For each $\alpha \in I$, choose a $w(\alpha) \in \{1,\ldots k\}$ such that $V_{\alpha}\subseteq U_{w(\alpha)}$ and define  for each $j \in \{1,\ldots k\}$,  \[Z_{j}=\bigcup_{\substack{\alpha \in I \\ w(\alpha)=j }}V_{\alpha}.\]  It easily follows that $\mathcal{Z} =\{Z_{j}\}_{j=1,\ldots,k}$ is a finite refinement of $\mathcal{U}$ satisfying $\dim(\mathcal{Z})\leq \dim(\mathcal V)$ and thus $\dim(\mathcal{Z})\leq \dim(X)$. This implies that $\dim_F(X)\leq \dim(X)$. Combining the (in)equalities above yields $\mathrm{dim}_{F}(X) =\mathrm{dim}_{LF}(X)=\mathrm{dim}(X)$.

Now, suppose $X$ is a separable metric space. Then $X$ is paracompact and Hausdorff, hence $\mathrm{dim}_{F}(X) =\mathrm{dim}_{LF}(X)=\mathrm{dim}(X)$.  Moreover, by Theorem 1.7.7 of \cite{Engelking}, we have $\mathrm{dim}_{F}(X)=\mathrm{ind}(X)=\mathrm{Ind}(X)$ (in \cite{Engelking}, the notation $\dim(X)$ is used for what we call $\dim_F(X)$ (see \cite[1.6.7]{Engelking})). This finishes the proof.
\end{proof}
The goal for the rest of this section is to prove the following generalization of Lemma 3.9 in \cite{Luck3}. 
\begin{proposition}\label{prop: G-map} Let $X$ be a separable metric space of topological dimension at most $n \geq 1$. Suppose the group $G$ acts isometrically and discretely on $X$. Then there exists a $G$-CW-complex $Y$ of dimension at most $n$ for which the stabilizers are subgroups of point stabilizers of $X$, together with a $G$-map $f: X \rightarrow Y$.
\end{proposition}
First, we need the following lemmas.
\begin{lemma}\label{lemma: metric space} If $G$ acts discretely and isometrically on a metric space $X$ then $G\setminus X$ inherits a metric from $X$, such that the metric topology on $G \setminus X$ coincides with the quotient topology.
\end{lemma}
\begin{proof} Let   $\pi: X \rightarrow G \setminus X$ be the natural quotient map and let $\pi(x)$, $\pi(y)  \in G \setminus X$. Define
\[    \overline{d}(\pi(x),\pi(y))=\inf\{ d(x,g\cdot y) \ | \ g \in G \} .    \]
We claim that $\overline{d}$ is a metric on $G \setminus X$. It is clearly a pseudo-metric. Now suppose that $\overline{d}(\pi(x),\pi(y))=0$. This means that there exists a sequence $\{g_n\}_{n \in \mathbb{N}}$ of elements in $G$ such that $\lim_{n \rightarrow \infty}d(x,g_n \cdot y)=0$.  By Lemma \ref{lemma: discrete}, there exists an $\varepsilon > 0$ such that for all $g \in G$ we have $g \cdot \mathrm{B}(y,\varepsilon)\cap \mathrm{B}(y,\varepsilon) \neq \emptyset \Leftrightarrow g \in G_{y}$. On the other hand, by the triangle inequality there exists an $N \in \mathbb{N}$ such that for all $n,m \geq N$, we have $d(y,g_n^{-1}g_m \cdot y)< \varepsilon$ and hence $g_n^{-1}g_m \in G_y$. Therefore, the sequence $\{g_n\cdot y\}_{n \in \mathbb{N}}$ has only finitely many distinct terms showing  that $d(x,g_n \cdot y)=0$ for some $n$. This implies that $\pi(x) = \pi(y) $, so $\overline{d}$ is a metric. It is an easy exercise  to check that the metric topology induced by this metric coincides  with the quotient topology induced by the map $\pi$.
\end{proof}

\begin{lemma}\label{lemma: top dim} 
Let $X$ be a separable metric space. If a group $G$ acts discretely and isometrically  on $X$, then $\mathrm{dim}(G \setminus X)= \mathrm{dim}(X)$.
\end{lemma}
\begin{proof}
Consider the quotient map $\pi: X \rightarrow G \setminus X$. By Lemma \ref{lemma: metric space}, $G\setminus X$ is a metric space such that the associated metric topology coincides with the quotient topology. Since $X$ is separable and continuous images of separable spaces are separable, $G \setminus X$ is also a separable metric space. Because $\pi$ is an open surjective map such that $G\cdot x$ is a discrete subset of $X$ for each $x \in X$, it follows from Theorem 1.12.7 in \cite{Engelking} that $\mathrm{ind}(G \setminus X)=\mathrm{ind}(X)$. From Lemma \ref{lemma: dim}, we deduce that $\mathrm{dim}(G \setminus X)= \mathrm{dim}(X)$.
\end{proof}
We can now prove Proposition \ref{prop: G-map}.
\begin{proof} Our arguments are  based on the proof of Lemma 3.9 in \cite{Luck3}. We will therefore  argue in detail where it suffices to have a discrete group action instead of a proper action and separable metric space instead of a proper metric space and refer to \cite{Luck3} for more details.

Suppose $\mathcal{V}$ is a G-invariant open cover of $X$ such that every $V \in \mathcal{V}$ satisfies the following condition: there exists a point $x_V \in X$ such that for each $g \in G$
\[ g\cdot V \cap V \neq \emptyset \Leftrightarrow g\cdot V = V \Leftrightarrow g \in G_{x_V}. \]
An open cover satisfying this condition will be called a \emph{good open cover}. The {\it nerve} $\mathcal{N}(\mathcal{V})$ of $\mathcal{V}$ is the simplicial complex whose vertices are the elements of $\mathcal{V}$ and the pairwise distinct vertices $V_0,\ldots,V_d$ span a $d$-simplex if and only if $\cap_{i=0}^d V_i \neq \emptyset$. Since $\mathcal{V}$ is G-invariant, the action of $G$ on $X$ induces a simplicial action of $G$ on $\mathcal{N}(\mathcal{V})$. Note that a $d$-simplex $(V_0,\ldots,V_d)$ is mapped to itself by a group element $g$  if and only if  for each $i \in \{0,\ldots,d\}$ there exists a  $j \in \{0,\ldots,d\}$ such that $g \cdot V_i = V_j$. Since $V_i \cap V_j \neq \emptyset$, we have $g \cdot V_i \cap V_i \neq \emptyset$ and hence $g \in G_{x_{V_i}}$. It follows that all vertices of the simplex $(V_0,\ldots,V_d)$ are fixed by $g$ and so the simplex  is fixed pointwise by $g$. Therefore, $\mathcal{N}(\mathcal{V})$ is a $G$-CW-complex for which the stabilizers are subgroups of point stabilizers of $X$. The $G$-CW-complex that appears in the statement of the proposition will be of this form. The aim is to find a G-invariant good open cover of $X$ that allows one to construct a $G$-map $f: X \rightarrow \mathcal{N}(\mathcal{V})$ and  satisfies $\mathrm{dim}(\mathcal{V})\leq n$. 

By discreteness of the action, for every $x \in X$ there exists   $\varepsilon(x)>0$ such that for every $g \in G$ we have 
\begin{align*} g \cdot {\mathrm{B}(x,2\varepsilon(x))} \cap  {\mathrm{B}(x,2\varepsilon(x))} \neq \emptyset &\Leftrightarrow  g \cdot {\mathrm{B}(x,2\varepsilon(x))}= {\mathrm{B}(x,2\varepsilon(x))}\\ 
&\Leftrightarrow g \cdot \mathrm{B}(x,\varepsilon(x)) = \mathrm{B}(x,\varepsilon(x))\\ &\Leftrightarrow g \in G_x 
\end{align*} 
and such that $\varepsilon(g\cdot x)=\varepsilon(x)$ holds for every $x \in X$ and every $g \in G$. Consider the  quotient map $\pi: X \rightarrow G \setminus X$. By Lemma \ref{lemma: top dim}, we have $\mathrm{dim}(G \setminus X)\leq n$. Note that $\{ \pi(\mathrm{B}(x,\varepsilon(x))) \ | \ x \in X \}$ is an open cover of $G \setminus X$. Since, by Lemma \ref{lemma: metric space}, $G \setminus X$ is a metric space, it is paracompact and Hausdorff. Therefore, by Lemma \ref{lemma: dim}, we can find a locally finite open covering $\mathcal{U}$ of $G \setminus X$ such that $\mathrm{dim}(\mathcal{U})\leq n$ and $\mathcal{U}$ is a refinement of $\{ \pi(\mathrm{B}(x,\varepsilon(x))) \ | \ x \in X \}$. Next, for each $U \in \mathcal{U}$, let $x_U \in X$ such that $U \subseteq  \pi(\mathrm{B}(x_U,\varepsilon(x_U)))$. Define the index set
\[  J=\{  (U,\overline{g}) \ | \ U \in \mathcal{U}, \overline{g} \in G/G_{x_U}               \},            \]
and for each $(U,\overline{g}) \in J$, define the open subset of $X$
\[ V_{U,\overline{g}}=   g \cdot  \mathrm{B}(x_U,2\varepsilon(x_U))  \cap  \pi^{-1}(U). \]
Then it follows that $\mathcal{V}=\{  V_{U,\overline{g}} \ | \ (U,\overline{g}) \in J   \}$  is a G-invariant good open cover of $X$ of $\mathrm{dim}(\mathcal{V})\leq n$.\\
\indent It remains to construct a $G$-map $f: X \rightarrow \mathcal{N}(\mathcal{V})$. To this end, take a locally finite partition of unity $\{e_U : G \setminus X \rightarrow [0,1] \ | \  U \in \mathcal{U}\}$ that is subordinate to $\mathcal{U}$. Fix a map $\chi: [0,\infty) \rightarrow [0,1]$ satisfying $\chi^{-1}(0)=[1,\infty)$ and define for each $(U,\overline{g}) \in J$ the function
\[ \phi_{U,\overline{g}}: X \rightarrow [0,1]: x \mapsto e_U(\pi(x))\chi\Big(d(x,g\cdot x_U))\Big/\varepsilon(x_U)\Big). \]
We claim that the collection $\{ \phi_{U,\overline{g}} \ | \ (U,\overline{g}) \in J  \}$ is locally finite. Let $y \in X$. Because $\mathcal{U}$ is locally finite,  we can find a $\delta>0$ such that $T=\mathrm{B}(\pi(y),\delta)$ intersects only finitely many elements of $\mathcal{U}$, say $U_1,\ldots U_m \in \mathcal{U}$. Let \[\varepsilon_0={\frac{1}{2}}\min\{  \varepsilon(x_{U_i}) \ | \ i \in \{1,\ldots,m \}\}\]and define
\[ W= \mathrm{B}(y,\varepsilon_0)\cap \pi^{-1}(T). \]
It follows that for each $i \in \{1,\ldots,m\}$, there exists  $g_i \in G$ such that    \[ \{ g \in G \ | \ {W} \cap  g \cdot  \mathrm{B}(x_{U_i},\varepsilon(x_{U_i}))\neq \emptyset\} \subseteq g_iG_{x_{U_i}} .\]   This shows that the set 
\[ J_W = \{  (U,\overline{g}) \in J \ | \ {W} \cap  g \cdot  \mathrm{B}(x_U,\varepsilon(x_U))  \cap  \pi^{-1}(U) \neq \emptyset \} \]
is finite. 

Suppose now  $x \in W$ and $\phi_{U,\overline{g}}(x)>0$. This implies that $$x\in {W} \cap  g \cdot  \mathrm{B}(x_U,\varepsilon(x_U))  \cap  \pi^{-1}(U)$$ which in turn shows that $(U,\overline{g}) \in J_W$. Since $J_W$ is finite, this proves the claim. 

It follows that the map
\[ \sum_{(U,\overline{g})\in J}  \phi_{U,\overline{g}} : X \rightarrow [0,1]: x \mapsto \sum_{(U,\overline{g})\in J}e_U(\pi(x))\chi\Big(d(x,g\cdot x_U)\Big/\varepsilon(x_U)\Big) \]
is well-defined and continuous. Moreover, one can check that this map has a value strictly greater than zero for every element $x \in X$. Define for each $(U,\overline{g}) \in J$, the map
\[ \psi_{U,\overline{g}}: X \rightarrow [0,1]: x \mapsto \frac{\phi_{U,\overline{g}}(x)}{  \sum_{(U,\overline{g})\in J}  \phi_{U,\overline{g}}(x)}. \]
Now, the map
\[ f: X \rightarrow \mathcal{N}(\mathcal{V}) : x \mapsto  \sum_{(U,\overline{g})\in J}  \psi_{U,\overline{g}}(x)V_{U,\overline{g}} \]
is the desired $G$-map. 
\end{proof}
\section{Bredon cohomology} \label{sec: bredon}
An important algebraic tool to study classifying spaces for families of subgroups is Bredon cohomology. This cohomology theory was introduced by Bredon in \cite{Bredon} for finite groups 
as a means to develop an obstruction theory for equivariant extension of maps.
It was later generalized to arbitrary groups by L\"{u}ck with applications to finiteness conditions  (see \cite[section 9]{Luck}, \cite{LuckMeintrup} and \cite{Luck1}).
Next, we  recall some basic notions of this theory. For more details, we refer the reader to \cite{Luck} and \cite{FluchThesis}. 

Let $G$ be a discrete group and let $\mathcal{F}$ be a family of subgroups of $G$. The \emph{orbit category} $\orb$ is the category defined by the objects which are the left coset spaces $G/H$ for all $H \in \mathcal{F}$ and the morphisms which are all $G$-equivariant maps between the objects. An \emph{$\orb$-module} is a contravariant functor $M: \orb \rightarrow \mathbb{Z}\mbox{-mod}$. The \emph{category of $\orb$-modules}, denoted by $\orbmod$, is defined by the objects which are all the $\orb$-modules and the morphisms which are all the natural transformations between these objects. A sequence \[0\rightarrow M' \rightarrow M \rightarrow M'' \rightarrow 0\]
in $\orbmod$ is called {\it exact} if it is exact after evaluating in $G/H$ for all $H \in \mathcal{F}$. Let $M \in \orbmod$ and consider the left exact functor
\[ \nathom(M,-) : \orbmod \rightarrow \mathbb{Z}\mbox{-mod}: N \mapsto \nathom(M,N), \]
where $\nathom(M,N)$ is the abelian group of all natural transformations from $M$ to $N$. The module $M$ is a \emph{projective $\orb$-module} if and only if this functor is exact. It can be shown that $\orbmod$ contains enough projective modules to construct projective resolutions. Hence, one can construct functors $\mathrm{Ext}^{n}_{\orb}(-,M)$ that have all the usual properties. The \emph{$n$-th Bredon cohomology of $G$} with coefficients $M \in \orbmod$ is by definition
\[ \mathrm{H}^n_{\mathcal{F}}(G,M)= \mathrm{Ext}^{n}_{\orb}(\underline{\mathbb{Z}},M), \]
where $\underline{\mathbb{Z}}$ is the functor that maps all objects to $\mathbb{Z}$ and all morphisms to the identity map. 

There is a notion of \emph{cohomological dimension of $G$ for the family $\mathcal{F}$}, denoted by $\mathrm{cd}_{\mathcal{F}}(G)$ and defined as
\[ \mathrm{cd}_{\mathcal{F}}(G) = \sup\{ n \in \mathbb{N} \ | \ \exists M \in \orbmod :  \mathrm{H}^n_{\mathcal{F}}(G,M)\neq 0 \}. \]
Using Shapiro's lemma for Bredon cohomology, one can show that $$\mathrm{cd}_{\mathcal{F}\cap H}(H) \leq  \mathrm{cd}_{\mathcal{F}}(G)$$ for any subgroup $H$ of $G$. Since the augmented cellular chain complex of any model for $E_{\mF}G$ yields a projective resolution of $\underline{\mathbb{Z}}$ which can then  be used to compute $\mathrm{H}_{\mF}^{\ast}(G,-)$, it follows that $ \mathrm{cd}_{\mathcal{F}}(G) \leq  \mathrm{gd}_{\mathcal{F}}(G)$. In  \cite[0.1]{LuckMeintrup}, L\"{u}ck and Meintrup prove  that one even has \[\mathrm{cd}_{\mathcal{F}}(G) \leq  \mathrm{gd}_{\mathcal{F}}(G) \leq \max\{3, \mathrm{cd}_{\mathcal{F}}(G) \}.\]

The following lemma shows that   it is possible to reduce the problem  of estimating the Bredon cohomological dimension of a countable group to its finitely generated subgroups. 
\begin{lemma} \label{lemma: finitely generated} Let $G$ be a countable group and let $\mathcal{F}$ be a family of subgroups of $G$ such that each subgroup in $\mathcal{F}$ is contained in a finitely generated subgroup of $G$. Suppose there exists an integer $d\geq 0$ such that $\mathrm{cd}_{\mathcal{F}\cap K}(K) \leq d$ for every finitely generated subgroup $K$ of $G$. Then we have $\mathrm{cd}_{\mathcal{F}}(G)\leq d+1$.
\end{lemma}
\begin{proof}
Denote by $\mathcal{SFG}$ the family of subgroups of finitely generated subgroups of $G$. Since $G$ can be written as a countable increasing union of finitely generated subgroups, 
we can construct a tree $T$ that is a one-dimensional model for $E_{\mathcal{SFG}}G$. Since $\mathcal{F}\subseteq \mathcal{SFG}$, we apply Corollary 4.1 of \cite{DemPetTal} to $T$ and obtain $\mathrm{cd}_{\mathcal{F}}(G)\leq d+1$.
\end{proof}

A chain complex of $\orb$-modules $C$ is said to be \emph{dominated} by a chain complex of  $\orb$-modules $D$ if there exists a chain map $i: C \rightarrow D$ and a chain map $r: D \rightarrow C$ such that $r\circ i$ is chain homotopy equivalent to the identity chain map on $C$. A chain complex of $\orb$-modules $C$ is called $d$-\emph{dimensional} if $C_{k}=0$ for all $k>d$.

The following proposition will come to use in the next section.
\begin{proposition}[{\cite[Proposition 11.10]{Luck}}]\label{prop: dominated} Let $P$ be a chain complex of projective $\orb$-modules. The following are equivalent for every integer $d\geq 0$.
\begin{itemize}
\item[(i)] The chain complex $P$ is dominated by a $d$-dimensional chain complex.
\smallskip
\item[(ii)] The chain complex $P$ is chain homotopy equivalent to a $d$-dimensional projective chain complex.
\smallskip
\item[(iii)] For every $\orb$-module $M$, we have $\mathrm{H}^{d+1}( \nathom(P,M))=0$ and for every integer $k>d$ and $H \in \mathcal{F}$, we have $\mathrm{H}_k(P(G/H))=0$.
\end{itemize}
\end{proposition}

A key ingredient for the proof of Theorem B is a general construction of  L\"{u}ck and Weiermann (see \cite{LuckWeiermann}) which relates Bredon cohomology  for a smaller family of subgroups to a larger one. Let us explain this construction  tailored to the case of the families of finite subgroups $\mathcal{F}$ and virtually cyclic subgroups $\mathcal{VC}$. Let $\mathcal S$ denote the set of infinite virtually cyclic subgroups of $G$. As in   \cite[2.2]{LuckWeiermann}, two infinite virtually cyclic subgroups $H$ and $K$ of $\Gamma$ are said to be \emph{equivalent}, denoted $H \sim K$, if $|H\cap K|=\infty$. Using the fact that any two infinite virtually cyclic subgroups of a virtually cyclic group are equivalent (e.g.~see Lemma 3.1. in \cite{DP}), it is easily seen that this indeed defines an equivalence relation on $\mathcal{S}$. One can also verify that this equivalence relation satisfies the following two properties
\begin{itemize}
\item[-]  $\forall H,K \in \mathcal{S} : H \subseteq K \Rightarrow H \sim K$;
\item[-] $ \forall H,K \in \mathcal{S},\forall g \in G: H \sim K \Leftrightarrow H^g \sim K^g$.
\end{itemize}
Let $H \in \mathcal{S}$ and define the group
\[     \mathrm{N}_{G}[H]=\{g \in G \ | \ H^g \sim H \}. \]
The group $ \mathrm{N}_{G}[H]$ is called the commensurator of $H$ in $G$ (also denoted by $\mathrm{Comm}_{G}[H]$). It depends only on the equivalence class $[H]$ of $H$. In particular, the commensurator of $H$ coincides with the commensurator of any infinite cyclic subgroup of $H$. Also note that $H$ is contained in $\mathrm{N}_{G}[H]$.
Denote the family of finite subgroups of $G$ by $\mathcal{F}$ and define for $H \in \mathcal{S}$ the following family of subgroups of $\mathrm{N}_{G}[H]$
\[ \mathcal{F}[H]=\{ K \subseteq \mathrm{N}_{G}[H] \ | K \in \mathcal{S}, K \sim H\} \cup \Big(\mathrm{N}_{G}[H] \cap \mathcal{F}\Big). \]
In other words,  $\mathcal{F}[H]$ contains all finite subgroups of $ \mathrm{N}_{G}[H]$ and all infinite virtually cyclic subgroups of $G$ that are equivalent to $H$. The pushout diagram in Theorem $2.3$ of \cite{LuckWeiermann} yields the following (see also \cite[\S 7]{DP2}).
\begin{proposition}[L\"{u}ck-Weiermann, \cite{LuckWeiermann}] \label{prop: push out} With the notation above, let $[\mathcal{S}]$ denote the set of equivalence classes of $\mathcal S$   and let $\mathcal{I}$ be a complete set of representatives $[H]$ of the orbits of the conjugation action of $G$ on $[\mathcal{S}]$. For every $M \in \mbox{Mod-}\mathcal{O}_{\mathcal{VC}}G$, there exists a long exact sequence
\[ \ldots \rightarrow \mathrm{H}^{i}_{\mathcal{VC}}(G,M) \rightarrow \Big(\prod_{[H] \in \mathcal{I}} \mathrm{H}^{i}_{\mathcal{F}[H] }( \mathrm{N}_{G}[H],M)\Big)\oplus  \mathrm{H}^{i}_{\mathcal{F}}(G,M) \rightarrow  \prod_{[H] \in \mathcal{I}} \mathrm{H}^{i}_{\mathcal{F}\cap  \mathrm{N}_{G}[H] }( \mathrm{N}_{G}[H],M)  \]
\[ \rightarrow \mathrm{H}^{i+1}_{\mathcal{VC}}(G,M) \rightarrow \ldots \ . \]
\end{proposition}

\section{Proofs of Theorems A and B}
Throughout this section, let $G$ be a discrete group and let $X$ be a CAT(0)-space on which $G$ acts by isometries.
We refer the reader to \cite{BridHaef} for the definition and properties of CAT(0)-spaces. 

We start by proving Theorem A. To this end, assume also that $G$ acts discretely on $X$ and that $X$ is separable  and of topological dimension $n$.
\begin{proof}[Proof of Theorem A]
By Proposition \ref{prop: G-map}, there exists an $n$-dimensional $G$-CW-complex $Y$ with stabilizers that are subgroups of point stabilizers of $X$ and a $G$-map $f: X \rightarrow Y$. Let $J_{\mathcal{F}}G$ be the terminal object in the $G$-homotopy category of $\mathcal{F}$-numerable $G$-spaces (see \cite[section 2]{Luck2}). We claim that there also exists a $G$-map $\varphi: E_{\mathcal{F}}G \rightarrow X \times J_{\mathcal{F}}G$. Assuming this, consider the following composition of $G$-equivariant maps 
\[  E_{\mathcal{F}}G  \xrightarrow{\varphi} X \times J_{\mathcal{F}}G \xrightarrow{f\times \mathrm{Id}}   Y \times  J_{\mathcal{F}}G \xrightarrow{\mathrm{Id}\times\alpha}  Y \times  E_{\mathcal{F}}G \xrightarrow{\pi_2}   E_{\mathcal{F}}G, \]
where $\alpha:  J_{\mathcal{F}}G  \rightarrow  E_{\mathcal{F}}G$ is the $G$-homotopy inverse of the $G$-homotopy equivalence $E_{\mathcal{F}}G  \rightarrow  J_{\mathcal{F}}G$ (see Theorem 3.7 of \cite{Luck2}), and $\pi_2$ is projection onto the second factor. We obtain G-maps such that their composition
\begin{equation} \label{eq: composition}        E_{\mathcal{F}}G \rightarrow  Y \times E_{\mathcal{F}}G \rightarrow E_{\mathcal{F}}G,               \end{equation}
 is $G$-homotopic to the identity map by the universal property of $E_{\mathcal{F}}G$. Using the equivariant cellular approximation theorem (see Theorem II.2.1 of \cite{tom}), we may assume that these maps are cellular and that their composition is $G$-homotopic via cellular maps to the identity map.\\
\indent Let $C_{\ast}( E_{\mathcal{F}}G)$ and $C_{\ast}(Y)$ be the associated cellular chain complexes of $E_{\mathcal{F}}G$ and $Y$. Then $C_{\ast}( E_{\mathcal{F}}G)\otimes C_{\ast}(Y)=C_{\ast}( E_{\mathcal{F}}G\times Y)$ where $C_{\ast}( E_{\mathcal{F}}G\times Y) $ is the cellular chain complex of $E_{\mathcal{F}}G\times Y$. The maps in (\ref{eq: composition}) induce $\mathcal{O}_{\mathcal{F}}G$-chain maps
\begin{equation*}        C_{\ast}(E_{\mathcal{F}}G) \rightarrow  C_{\ast}(E_{\mathcal{F}}G \times Y)  \rightarrow C_{\ast}(E_{\mathcal{F}}G),               \end{equation*}
such that the composition is chain homotopy equivalent to the identity chain map on $ C_{\ast}(E_{\mathcal{F}}G)$. Note that $C_{\ast}( E_{\mathcal{F}}G\times Y)$ is a chain complex of free $\mathcal{O}_{\mathcal{F}}G$-modules and that $C_{\ast}( E_{\mathcal{F}}G) \rightarrow \underline{\mathbb{Z}}$ is a free resolution of $\underline{\mathbb{Z}}$. In the proof of Proposition 3.2 of \cite{DemPetTal}, the authors construct a convergent spectral sequence
\[ E_{1}^{p,q}= \prod_{\sigma \in \Sigma_p}\mathrm{H}^q_{\mathcal{F}\cap G_{\sigma}}(G_{\sigma},M) \Rightarrow \mathrm{H}^{p+q}(\mathrm{Hom}_{\mathcal{F}}( C_{\ast}(E_{\mathcal{F}}G)\otimes C_{\ast}(Y),M)) \]
for every  $\orb$-module $M$, where $\Sigma_p$ is a set of  representatives of all the $G$-orbits of  $p$-cells of $Y$ and $G_{\sigma}$ is the stabilizer of $\sigma$. Since $Y$ is $n$-dimensional and  $\mathrm{cd}_{\mathcal{F}\cap G_{\sigma}}(G_{\sigma})\leq d$ for each $\sigma$, we conclude from the spectral sequence that $$\mathrm{H}^{n+d+1}(\mathrm{Hom}_{\mathcal{F}}( C_{\ast}( E_{\mathcal{F}}G\times Y),M))=0$$ for every  $\orb$-module $M$. Also, for each $H \in \mathcal{F}$, the complex $C_{\ast}(E_{\mathcal{F}}G)(G/H)\twoheadrightarrow \mathbb{Z}$ is exact and $C_{k}(Y)(G/H)$ is $\mathbb{Z}$-free for $k \geq 0$. Combining these observations with the fact that $C_{\ast}(Y)$ is $n$-dimensional, a double complex spectral sequence argument shows that $$\mathrm{H}_k( C_{\ast}( E_{\mathcal{F}}G\times Y)(G/H))=0$$ for every $H \in \mathcal{F}$ and every $k > n$.
We conclude from Proposition \ref{prop: dominated} that $ C_{\ast}( E_{\mathcal{F}}G\times Y)$ is chain homotopy equivalent to an $(n+d)$-dimensional projective chain complex $Z$. Therefore, there exist chain maps $i: C_{\ast}( E_{\mathcal{F}}G) \rightarrow Z $ and $r: Z \rightarrow C_{\ast}( E_{\mathcal{F}}G)$ such that $r \circ i$ is chain homotopy equivalent to the identity chain map on $C_{\ast}( E_{\mathcal{F}}G)$. Hence  $C_{\ast}( E_{\mathcal{F}}G)$ is dominated by an $(n+d)$-dimensional chain complex $Z$. It follows from Proposition \ref{prop: dominated} that $C_{\ast}( E_{\mathcal{F}}G)$ is chain homotopy equivalent to an $(n+d)$-dimensional projective chain complex $P$. But then $P \rightarrow \underline{\mathbb{Z}}$ is a projective $(n+d)$-dimensional $\mathcal{O}_{\mathcal{F}}G$-resolution of $\underline{\mathbb{Z}}$. This implies that $\mathrm{cd}_{\mathcal{F}}(G)\leq n+d$. \\
\indent It remains to prove the claim that there exists a $G$-map $\varphi: E_{\mathcal{F}}G \rightarrow X \times J_{\mathcal{F}}G$. It suffices to show that $ X \times J_{\mathcal{F}}G$ is a model for $J_{\mathcal{F}}G$. The existence of the desired map will then follow from the universal property of $J_{\mathcal{F}}G$, since $E_{\mathcal{F}}G$ is an $\mathcal{F}$-numerable $G$-space. 

A standard fact which follows directly from the definition of an $\mathcal{F}$-numerable $G$-space is that if there is a $G$-map between $G$-spaces with an $\mathcal{F}$-numerable  target then the source is also $\mathcal{F}$-numerable. Hence, $X \times J_{\mathcal{F}}G$ is an $\mathcal{F}$-numerable $G$-space by virtue of the projection onto the second coordinate $X \times J_{\mathcal{F}}G \rightarrow J_{\mathcal{F}}G$. We equip the product $(X \times J_{\mathcal{F}}G) \times (X \times J_{\mathcal{F}}G)$ with the diagonal $G$-action and denote the projection of 
$(X \times J_{\mathcal{F}}G) \times (X \times J_{\mathcal{F}}G)$ onto the $i$-th factor $X \times J_{\mathcal{F}}G$ by $\mathrm{pr}_i$, for $i=1,2$. By Theorem 2.5(ii) of \cite{Luck2}, $X \times J_{\mathcal{F}}G$ is a model for $J_{\mathcal{F}}G$ if and only if each $H \in \mathcal{F}$ is contained in a point stabilizer of $X \times J_{\mathcal{F}}G$ and $\mathrm{pr}_1$ and $\mathrm{pr}_2$ are $G$-homotopic. Let $G$ act diagonally on the product $J_{\mathcal{F}}G \times  J_{\mathcal{F}}G$ and denote the projection of 
$J_{\mathcal{F}}G \times J_{\mathcal{F}}G$ onto the $i$-th factor $J_{\mathcal{F}}G$ by $\mathrm{p}_i$, for $i=1,2$. It follows from Theorem 2.5(ii) of \cite{Luck2} that each subgroup in $\mathcal{F}$ is contained in a point stabilizer of $J_{\mathcal{F}}G$ and that $\mathrm{p}_1$ and $\mathrm{p}_2$ are $G$-homotopic via a $G$-map $Q: I \times J_{\mathcal{F}}G \times J_{\mathcal{F}}G \rightarrow J_{\mathcal{F}}G$ with $Q_{|t=0}=\mathrm{p}_1$ and $Q_{|t=1}=\mathrm{p}_2$. Let $H \in \mathcal{F}$. Then $(X \times J_{\mathcal{F}}G)^H=X^H \times J_{\mathcal{F}}G^H$ is non-empty. Hence, $H$ is contained in a point stabilizer of $X \times J_{\mathcal{F}}G$.  Let $G$ act diagonally on $X \times X$  and denote the projection of $X \times X$ onto the $i$-th factor $X$ by $\mathrm{q}_i$, for $i=1,2$. By Proposition II.1.4 of  \cite{BridHaef}, each pair of points $x,y \in X$ can be joined by a unique geodesic, i.e.~an isometry $\gamma_{x,y}: [0,L] \rightarrow X$ with $\gamma_{x,y}(0)=x$, $\gamma_{x,y}(L)=y$, and this geodesic varies continuously with its endpoints. By rescaling  $\gamma_{x,y}$ to a map $\gamma_{x,y}': [0,1] \rightarrow X$, we obtain a $G$-map
\[ R:I\times X \times X \rightarrow X : (t,x,y)\mapsto \gamma_{x,y}'(t)\]
with  $R_{|t=0}=\mathrm{q}_1$ and $R_{|t=1}=\mathrm{q}_2$.  But then the map
\begin{eqnarray*}
S: I \times (X \times J_{\mathcal{F}}G) \times (X \times J_{\mathcal{F}}G) & \rightarrow & X \times J_{\mathcal{F}}G :\\
(t,x,a,y,b)&\mapsto & (R(t,x,y),Q(t,a,b))      
\end{eqnarray*}
is a $G$-map with $S_{|t=0}=\mathrm{pr}_1$ and $S_{|t=1}=\mathrm{pr}_2$. This proves the claim and finishes the proof.
\end{proof}
From this point onward, we impose the additional assumption that the CAT(0)-space $X$ is complete. 

Before turning to the proof of Theorem B, let us first recall the definitions of semi-simple, elliptic and hyperbolic isometries. Let $g$ be an element of $G$, and hence an isometry of $X$. The \emph{displacement function} $d_{g}$ is the function
\[ d_{g}: X \rightarrow \mathbb{R}^{+}: x \mapsto d(g \cdot x,x). \]
The \emph{translation length} of $g$ is the number
\[ |g|=\inf\{d_{g}(x) \ | \ x \in X\}. \]
The space $\mathrm{Min}(g)$ is the set of all points $x \in X$ for which $d_{g}(x)=|g|$. The element $g$ is called \emph{semi-simple} if $\mathrm{Min}(g)$ is non-empty. In this case, $\mathrm{Min}(g)$ is a closed convex subset of $X$ and therefore also a complete CAT(0)-space (see Proposition II.6.2(3) of \cite{BridHaef}). 

The following proposition extends Proposition II.6.10(2) of  \cite{BridHaef} to discrete  actions.
\begin{proposition}\label{prop: cocompact} If the group $G$ acts discretely and cocompactly on a proper metric space $M$, then every element of $G$ acts as a semi-simple isometry.
\end{proposition}
\begin{proof} By cocompactness, we can find a compact subset $K$ of $M$ such that $\bigcup_{g \in G}g\cdot K = M$. Fix $g \in G$ and let $\{x_n\}$ be a sequence of elements in $M$ such that $\lim_{n \rightarrow \infty}d(g\cdot x_n,x_n)=|g|$. Choose a sequence of elements $\{y_n\}$ in $K$ such that there exists a sequence of group elements $\{g_n\}$ for which $g_n \cdot x_n = y_n$, for each $n\geq 0$. By the compactness of $K$, we can pass to subsequences if necessary and assume that $\lim_{n\rightarrow \infty}y_n = y \in K$. Denote $h_n=g_ngg_n^{-1}$ and note that $d(h_n\cdot y_n,y_n)=d(g\cdot x_n,x_n)$, for each $n\geq 0$. Since,
\begin{eqnarray*}
0 \leq d(h_n\cdot y,y)-|g| &\leq & d(h_n \cdot y,h_n \cdot y_n) +d(h_n\cdot y_n,y)- |g| \\
                      &\leq & d(h_n \cdot y,h_n \cdot y_n) +d(h_n\cdot y_n,y_n)+d(y_n,y)- |g| \\
                      &\leq & 2d(y,y_n)+d(g\cdot x_n,x_n)-|g| 
\end{eqnarray*}
for each $n$, we have $\lim_{n \rightarrow \infty} d(h_n\cdot y,y)=|g|$. 
It follows that the sequence $\{h_n \cdot y\}$ is contained in a closed ball centered at $y$. Hence, by passing to subsequences if necessary, we may assume that $ \{h_n\cdot y\}$ converges. By discreteness of the action, this implies that the sequence $\{h_n \cdot y\}$ has only finitely many distinct values. Again by passing to a subsequence, we may assume that $h_n\cdot y=h_0\cdot y$ for all $n$. We now have
\[ d(gg_0^{-1}\cdot y,g_0^{-1}\cdot y)= d(h_0\cdot y, y)= \lim_{n \rightarrow \infty} d(h_n\cdot y,y)=|g|. \]
This means that $g_0^{-1}\cdot y \in \mathrm{Min}(g)$, hence $g$ is a semi-simple isometry. 
\end{proof}

\begin{definition} \rm Let $g \in G$ be semi-simple. The element $g$, viewed as an isometry of $X$, is called
\begin{itemize}
\item[-] \emph{elliptic} if $|g|=0$, i.e.~$g$ has a fixed point;
\item[-] \emph{hyperbolic} if $|g|>0$, i.e.~$g$ has no fixed point.
\end{itemize}
\end{definition}
It follows that if  $g^m \in G$ is elliptic (hyperbolic) for some integer $m \neq 0$, then $g$ is also elliptic (hyperbolic) (see Proposition II.6.7 and Theorem II.6.8(2) of \cite{BridHaef}).

\begin{proposition}\label{prop: first one} Let $G$ be a discrete group and let $X$ be a complete separable CAT(0)-space of topological dimension $n$ on which $G$ acts discretely and isometrically. Let $\mathcal{F}$ be the family of finite subgroups of $G$.
If $H$ is an infinite cyclic subgroup of $G$ generated by an elliptic element, then \[\mathrm{cd}_{\mathcal{F}[H]}(\mathrm{N}_{G}[H])\leq n+\max\{\ms(G, X),\mms(G, X)\}.\]
\end{proposition}
\begin{proof} We claim that $X^{S}$ is non-empty and contractible, for every $S \in \mathcal{F}[H]$. First of all, Corollary II.2.8(1) of \cite{BridHaef} gives us that $X^{F}$ is non-empty and convex for every finite subgroup of $\mathrm{N}_{G}[H]$. So, let $K$ be an infinite virtually cyclic subgroup of $\mathrm{N}_{G}[H]$ that is equivalent to $H$. Let $h$ be an elliptic element of $G$ that generates $H$. Because $H \sim K$, there exists a non-zero integer $m$ such that $\langle h^m \rangle$ is a finite index subgroup of $K$. Since $h$ is elliptic, we can find an $x \in X$ such that $h\cdot x = x$. But then the orbit of $x$ under the action of $K$ is finite. It follows from Corollary II.2.8(1) of \cite{BridHaef} that $X^{K}$ is non-empty and convex. This proves our claim.  

Now,  Theorem A applied to the group $\mathrm{N}_{G}[H]$ acting on $X$ and the family $\mathcal{F}[H]$  implies that $\mathrm{cd}_{\mathcal{F}[H]}(\mathrm{N}_{G}[H]) \leq n + p$, where
\[ p=\max\{  \mathrm{cd}_{\mathcal{F}[H]\cap \mathrm{N}_{G}[H]_x}(\mathrm{N}_{G}[H]_x)   \ | \ x \in X      \}.\]
To prove the proposition, it remains to show that $p \leq \max\{\ms(G, X), \mms(G, X)\}$. 

Let $x \in X$. We distinguish between two cases. 
First, assume that $\mathcal{F}[H]\cap \mathrm{N}_{G}[H]_x$ coincides with the family of finite subgroups of $ \mathrm{N}_{G}[H]_x$. In this case, we have 
$$\mathrm{cd}_{\mathcal{F}[H]\cap \mathrm{N}_{G}[H]_x}(\mathrm{N}_{G}[H]_x) =\underline{\mathrm{cd}}( \mathrm{N}_{G}[H]_x)\leq \underline{\mathrm{cd}}(G_x) \leq \ms(G, X).$$ Secondly, assume that $\mathcal{F}[H]\cap \mathrm{N}_{G}[H]_x$ is different from the family of finite subgroups of $ \mathrm{N}_{G}[H]_x$. Clearly, it contains the family of finite subgroups of $ \mathrm{N}_{G}[H]_x$ and by assumption it must also contain an infinite cyclic subgroup $K$ that is equivalent to $H$. Since $\mathrm{N}_{G}[H]$ only depends on $H$ up to equivalence, we may as well assume that $H$ is contained in $\mathrm{N}_{G}[H]_x$. But then, we have
\[  \mathrm{N}_{G}[H]_x=  \mathrm{N}_{G}[H]\cap G_x =\mathrm{N}_{G_x}[H]\]
and  $\mathrm{cd}_{\mathcal{F}[H]\cap \mathrm{N}_{G}[H]_x}(\mathrm{N}_{G}[H]_x)= \mathrm{cd}_{(\mathcal{F}\cap\mathrm{N}_{G_x}[H]) [H]}(\mathrm{N}_{G_x}[H])$. It now follows from Lemma 6.3 of \cite{DP2} that \[ \mathrm{cd}_{\mathcal{F}[H]\cap \mathrm{N}_{G}[H]_x}(\mathrm{N}_{G}[H]_x)\leq  \max\{\underline{\mathrm{cd}}(\mathrm{N}_{G_x}[H]), \underline{\underline{\mathrm{cd}}}(\mathrm{N}_{G_x}[H])\},\]
which implies that 
\[ \mathrm{cd}_{\mathcal{F}[H]\cap \mathrm{N}_{G}[H]_x}(\mathrm{N}_{G}[H]_x)\leq \max\{ \underline{\mathrm{cd}}(G_x),\underline{\underline{\mathrm{cd}}}(G_x)\}. \]
This proves that $p \leq  \max\{\ms(G, X), \mms(G, X)\}$.
\end{proof}
Before proceeding, we refer to Theorem II.6.8 in \cite{BridHaef} for the definition and basic properties of an axis of a hyperbolic element.
\begin{proposition}\label{prop: second one} Let $G$ be a countable discrete group and let $X$ be a complete separable CAT(0)-space of topological dimension $n$  on which $G$ acts discretely and isometrically. Let $H$ be an infinite cyclic subgroup of $G$ generated by a hyperbolic element, then \[\mathrm{cd}_{\mathcal{F}[H]}(\mathrm{N}_{G}[H])\leq n+\mv(G, X).\]
\end{proposition}
\begin{proof}
Let $h$ be a hyperbolic element of $G$ that generates $H$ and let $g \in \mathrm{N}_{G}[H]$. By definition, there exist non-zero integers $l$ and $m$ such that $g^{-1}h^lg=h^m$. Proposition II.6.2(2) of \cite{BridHaef} implies that $|h^l|=|h^m|$. By applying Proposition II.6.2(4) and Theorem II.6.8(1) of \cite{BridHaef} to an axis of $h$, it follows that $|h^k|=\pm|h|k$ for all $k \in \mathbb{Z}$. We deduce that $l=\pm m$. 

Next, let $K$ be a finitely generated subgroup of $\mathrm{N}_{G}[H]$ that contains $H$. We can find a non-zero integer $m$ such that $g^{-1}h^mg=h^{\pm m}$ for all $g \in K$. By replacing $H=\langle h \rangle $ with $H=\langle h^m \rangle$, we may assume that $m=1$. Hence, $H$ is a normal subgroup of $K$ and $\mathrm{cd}_{\mathcal{F}[H]\cap K}(K)= \underline{\mathrm{cd}}(K/H)$ by Lemma 4.2 of  \cite{DP2}. 

Since $g^{-1}hg=h^{\pm 1}$ for all $g \in K$ and $\mathrm{Min}(h)=\mathrm{Min}(h^{-1})$, it follows from Proposition II.6.2(2) of \cite{BridHaef} that $K$ acts on $\mathrm{Min}(h)$. Moreover, $K$ maps an axis of $h$ to an axis of $h$. 

It follows from Theorem II.2.14 and Proposition I.5.3(4) of \cite{BridHaef} that there exists a complete separable $\mathrm{CAT}(0)$-subspace $Y$ of $X$ such that $\mathrm{Min}(h)$ is isometric to $Y \times \mathbb{R}$ and $K$ acts on $\mathrm{Min}(h)=Y \times \mathbb{R}$ via discrete isometries in $\mathrm{Iso}(Y)\times \mathrm{Iso}(\mathbb{R})$. Since $H$ acts by non-trivial translations on each axis, it acts identically on  $Y$ via the projection of $\mathrm{Iso}(Y)\times \mathrm{Iso}(\mathbb{R})$ onto $\mathrm{Iso}(Y)$, and it acts cocompactly on $\mathbb{R}$  via the projection of $\mathrm{Iso}(Y)\times \mathrm{Iso}(\mathbb{R})$ onto $\mathrm{Iso}(\mathbb{R})$. 

Let $x \in Y$ and consider the point stabilizer $K_x$ of the action of $K$ on $Y$ given by the projection of $\mathrm{Iso}(Y)\times \mathrm{Iso}(\mathbb{R})$ onto $\mathrm{Iso}(Y)$. The projection of $\mathrm{Iso}(Y)\times \mathrm{Iso}(\mathbb{R})$ onto $\mathrm{Iso}(\mathbb{R})$, maps $K_x$  onto a discrete subgroup of $\mathrm{Iso}(\mathbb{R})$ that acts discretely and cocompactly on $\mathbb{R}$, i.e. onto a subgroup $Q$ of the infinite dihedral group $D_{\infty}$. The kernel of this map is a subgroup $N$ of the point stabilizer $G_x$. Hence, we have a short exact sequence
\[   1 \rightarrow N \rightarrow K_x \rightarrow Q \rightarrow 1.      \]
Since $H$ is a normal subgroup of $K$ that acts identically on $Y$ and cocompactly on $\mathbb{R}$, by Lemma \ref{lemma: product}, $K/H$ acts isometrically and discretely  on $Y$ which is a  complete separable CAT(0)-space. Moreover, by Theorem 2 of \cite{Morita}, the topological dimension of $Y$ is at most $n-1$. The  point stabilizers  $(K/H)_x$ of the action of $K/H$ on $Y$ are of the form
\[   1 \rightarrow N \rightarrow (K/H)_x \rightarrow Q/H \rightarrow 1,     \]
where $N$ is a subgroup $G_x$ and $Q/H$ is a subgroup of a finite dihedral group. It follows from Theorem A applied to the family of finite subgroups of $K/H$ that $\underline{\mathrm{cd}}(K/H)\leq n-1+\mv(G, X)$. Hence, we have $\mathrm{cd}_{\mathcal{F}[H]\cap K}(K)\leq n-1+\mv(G, X)$  and therefore, by Lemma \ref{lemma: finitely generated}, $\mathrm{cd}_{\mathcal{F}[H]}(\mathrm{N}_{G}[H])\leq n+\mv(G, X)$.
\end{proof}
We can now prove Theorem B.
\begin{proof}[Proof of Theorem B]  Let $\mF$ be the family of finite subgroups of $G$ , let $\mathcal{VC}$ be the family of virtually cyclic subgroups of $G$ and denote by $\mathcal S$ the set of infinite virtually cyclic subgroups of $G$. Let $[H]$ be the equivalence class represented by $H \in \mathcal{S}$ and let $\mathcal{I}$ be a complete set of representatives $[H]$ of the orbits of the conjugation action of $G$ on $[\mathcal{S}]$. Note that we may assume that each class $[H]$ in $\mathcal{I}$ is represented by an infinite cyclic group $H$, generated by either an elliptic element or a hyperbolic element. By Proposition \ref{prop: push out}, for every $M \in \mbox{Mod-}\mathcal{O}_{\mathcal{VC}}G$ we have
\[ \ldots \rightarrow \mathrm{H}^{i}_{\mathcal{VC}}(G,M) \rightarrow \Big(\prod_{[H] \in \mathcal{I}} \mathrm{H}^{i}_{\mathcal{F}[H] }( \mathrm{N}_{G}[H],M)\Big)\oplus  \mathrm{H}^{i}_{\mathcal{F}}(G,M) \rightarrow  \prod_{[H] \in \mathcal{I}} \mathrm{H}^{i}_{\mathcal{F}\cap  \mathrm{N}_{G}[H] }( \mathrm{N}_{G}[H],M)  \]
\[ \rightarrow \mathrm{H}^{i+1}_{\mathcal{VC}}(G,M) \rightarrow \ldots \ . \]
\indent By the preceding two propositions, we have $$\mathrm{cd}_{\mathcal{F}[H]}(\mathrm{N}_{\Gamma}[H])\leq n+ \max\{\mms(G, X), \mv(G, X)\}$$
 for each $[H] \in \mathcal{I}$. Also,  from Theorem A we can deduce that $$\mathrm{cd}_{\mathcal{F}\cap  \mathrm{N}_{\Gamma}[H]  }(\mathrm{N}_{\Gamma}[H])\leq \mathrm{cd}_{\mathcal{F}}(G) \leq n+\ms(G, X),$$ for each $[H] \in \mathcal{I}$. It then follows from the long exact cohomology sequence that $$\mathrm{H}^{i}_{\mathcal{VC}}(G,M)=0 \;\; \mbox{ for all }\;\; i > n+\max\{\mms(G, X), \mv(G, X)+1\}$$ and for every $M \in \mbox{Mod-}\mathcal{O}_{\mathcal{VC}}G$, which finishes  the proof. 
\end{proof}
\section{Applications}
\subsection{Actions with specific point stabilizers} Let us first recall the following facts. By \cite[1.2]{Dunwoody}, a virtually free group $G$ acts simplicially on a simplicial tree with finite stabilizers. Hence, by Corollary \ref{cor: finite}, we have $\underline{\mathrm{cd}}(G)\leq 1$  and by Theorem B,  $\underline{\underline{\mathrm{cd}}}(G)\leq 2$. 

Next, let $G$ be a virtually polycyclic group of Hirsch length $h$. In \cite[5.26]{Luck2} it is shown that $\underline{\mathrm{cd}}(G)= h$, and in \cite[5.13]{LuckWeiermann} it is proven that $\underline{\underline{\mathrm{cd}}}(G)\leq h+1$. 

Finally, let $G$ be a countable elementary amenable group of Hirsch length $h$. From \cite{FloresNuc}, it follows that $\underline{\mathrm{cd}}(G)\leq h+1$, and in \cite[Theorem A]{DP2}, the authors show that $\underline{\underline{\mathrm{cd}}}(G)\leq h+2$. Also note that the classes of virtually free, virtually polycyclic and elementary amenable groups are closed under taking subgroups and finite extensions.

Now, suppose $G$ is a group that acts on a topological space $X$. Then, by the above stated results, it readily follows that
\[\mv(G, X)  \leq 1 \ \;\;\; \mbox{and} \;\;\; \   \mms(G, X)\leq 2\]
if $G_x$
is virtually free for each $x \in X$,
\smallskip

\[\mv(G, X)  \leq  h   \ \;\;\; \mbox{and} \;\;\; \ \ \mms(G, X)  \leq h+1\]
if $G_x$
is virtually polycyclic of Hirsch length at most $h$ for each $x \in X$, and
\smallskip

\[\mv(G, X)   \leq h+1  \  \;\;\; \mbox{and} \;\;\;  \ \mms(G, X)  \leq h+2\]
if $G_x$
is elementary amenable of Hirsch length at most $h$ for each $x \in X$.
\smallskip

Using Corollary \ref{cor: finite} and Theorem B, we deduce Corollary \ref{cor: intro specific stab}.
\subsection{Generalized Baumslag-Solitar groups} Generalized Baumslag-Solitar groups are fundamental groups of finite graphs of groups where all vertex and edge groups are infinite cyclic. In \cite{Kropholler}, P. Kropholler characterized non-cyclic generalized Baumslag-Solitar groups as those finitely generated groups that are of cohomological dimension two and have an infinite cyclic subgroup that intersects all of its conjugates non-trivially. Clearly, the well-known Baumslag-solitar groups 
\[BS(m,n)=\langle x,y \ | \ xy^{m}x^{-1}=y^{n} \rangle \]
are examples of generalized Baumslag-Solitar groups. 

Given an arbitrary generalized Baumslag-Solitar group $G$, we will determine $\underline{\underline{\mathrm{cd}}}(G)$. The result will depend on whether or not $G$ contains a copy of $\mathbb{Z}^2$. 
We will need the following  property of the generalized Baumslag-Solitar groups (see \cite[2.5]{Forester}).
\begin{lemma} \label{lem: elliptic} Suppose $G$ acts on a simplicial tree such that all vertex and edge stabilizers are infinite cyclic. Then any two elliptic elements of $G$ generate equivalent infinite cyclic subgroups of $G$, and $\mathrm{N}_G[\langle h \rangle]=G$ for every elliptic element $h \in G$.
\end{lemma}
\begin{proof} Let $g_1$ and $g_2$ be  elliptic elements of $G$. By definition, there exist vertices $v_1,v_2 \in T$ such that $g_1 \cdot v_1=v_1$ and  $g_2 \cdot v_2=v_2$. 
First, suppose that $v_1=v_2$. In this case, $g_1$ and $g_2$ are both contained in the stabilizers of $v_1$, which is infinite cyclic. This implies that $\langle g_1 \rangle$ and $\langle g_2 \rangle$ are equivalent. Now, consider the case where $v_1$ and $v_2$ are connected by an edge $e$. Denote the stabilizer of $e$ by $H$. Since $H$ is a finite index subgroup of the stabilizers of $v_1$ and $v_2$, it follows that $H$ is equivalent to  $\langle g_1 \rangle$ and $\langle g_2 \rangle$. Hence, $\langle g_1 \rangle$ and $\langle g_2 \rangle$ are equivalent. The general case now follows by induction on the distance between $v_1$ and $v_2$. \\
\indent Since any conjugate of an elliptic element is again elliptic and two elliptic elements of $G$ generate equivalent infinite cyclic subgroups, it follows that $\mathrm{N}_G[\langle h \rangle]=G$ for every elliptic element $h \in G$.
\end{proof}

\begin{lemma} \label{lemma: free group} Let $F$ be a finitely generated 
non-abelian free group, then $\underline{\underline{\mathrm{cd}}}(F)=2$.
\end{lemma}
\begin{proof} Since $F$ acts freely on a tree, we have 
$\underline{\underline{\mathrm{cd}}}(F)\leq 2$ by Corollary \ref{cor: intro specific stab}. Let $\underline{\mathbb{Z}}$ be the trivial right 
$\mathcal{O}_{\mathcal{VC}}F$-module  and let $\mathcal{TR}$ be the 
family of $F$ containing only the trivial subgroup. Proposition \ref{prop: 
push out} yields the following exact sequence
\[ \Big(\prod_{[H] \in \mathcal{I}} \mathrm{H}^{1}_{\mathcal{TR}[H] }( 
\mathrm{N}_{F}[H],\underline{\mathbb{Z}})\Big)\oplus 
\mathrm{H}^{1}(F,\mathbb{Z}) \rightarrow  \prod_{[H] \in \mathcal{I}} 
\mathrm{H}^{1}( \mathrm{N}_{F}[H],\mathbb{Z}) \rightarrow 
\mathrm{H}^{2}_{\mathcal{VC}}(F,\underline{\mathbb{Z}}) . \]
Since $F$ is a free group, it follows  that $\mathrm{N}_{F}[H]\cong 
\mathbb{Z}$ for each $H \in \mathcal{I}$. Hence, we have 
$\mathrm{cd}_{\mathcal{TR}[H]}(N_F[H])=0$ for each $H \in \mathcal{I}$, 
and the exact sequence reduces to
\[\mathrm{H}^{1}(F,\mathbb{Z}) \rightarrow  \prod_{[H] \in \mathcal{I}} 
\mathrm{H}^{1}( \mathbb{Z},\mathbb{Z}) \rightarrow 
\mathrm{H}^{2}_{\mathcal{VC}}(F,\underline{\mathbb{Z}}) . \]
Since the set $\mathcal{I}$ is infinite  and 
$\mathrm{H}^{1}(F,{\mathbb{Z}}) $ is finitely generated, it 
follows that $\mathrm{H}^{2}_{\mathcal{VC}}(F,\underline{\mathbb{Z}}) 
\neq 0$.
We conclude that $\underline{\underline{\mathrm{cd}}}(F)=2$.
\end{proof}

\begin{proof}[Proof of Corollary \ref{cor: intro baumslag}] Clearly, if $G\cong \Z$ then $\underline{\underline{\mathrm{cd}}}(G)= 0$. Let us suppose now that $G$ is not cyclic.  By Corollary \ref{cor: intro specific stab}, it follows that
$\underline{\underline{\mathrm{cd}}}(G)\leq 3$. If $G$ contains a subgroup isomorphic to 
$\mathbb{Z}^2$ then $\underline{\underline{\mathrm{cd}}}(G)= 3$, 
since $\underline{\underline{\mathrm{cd}}}(\mathbb{Z}^2)= 3$ (see 
e.g.~\cite[5.21]{LuckWeiermann}). 

Suppose $G$ does 
not contain a subgroup isomorphic to $\mathbb{Z}^2$. By definition, $G$ 
acts on a tree $T$ with infinite cyclic edge and vertex stabilizers. Let 
$H$ be an infinite cyclic subgroup of $G$ generated by a hyperbolic 
element $h$ of $G$. By Proposition 24 in \cite{Serre}, the axis of $h$ 
coincides with the axis of $h^i$, for each $i\geq 1$. Using this fact, it is not difficult to
check that $\mathrm{N}_G[H]$ acts on the axis of $h$. As in the proof 
of Proposition \ref{prop: second one}, the action of $\mathrm{N}_G[H]$ 
on the axis of $h$ yields a short exact sequence
\[ 1 \rightarrow N \rightarrow \mathrm{N}_G[H] \rightarrow Q \rightarrow 
1 \]
such that $N$ is a subgroup of $\mathbb{Z}$ and $Q$ is an infinite 
subgroup of the infinite dihedral group. Since $G$ does not contain a 
subgroup isomorphic to $\mathbb{Z}^2$, this implies that $N$ is  
trivial. Since $G$ is torsion-free, it follows that $ 
\mathrm{N}_G[H]\cong \mathbb{Z}$. \\
\indent Let $M$ be a right $\mathcal{O}_{\mathcal{VC}}G$-module an let 
$\mathcal{TR}$ be the family of $G$ containing only the trivial 
subgroup. By Proposition \ref{prop: push out}, we have the following long exact 
sequence
\[  \mathrm{H}^{1}_{\mathcal{VC}}(G,M) \rightarrow \Big(\prod_{[H] \in 
\mathcal{I}} \mathrm{H}^{1}_{\mathcal{TR}[H] }( 
\mathrm{N}_{G}[H],M)\Big)\oplus  \mathrm{H}^{1}(G,M) \rightarrow 
\prod_{[H] \in \mathcal{I}} \mathrm{H}^{1}( \mathrm{N}_{G}[H],M) \]
\[  \rightarrow \mathrm{H}^{2}_{\mathcal{VC}}(G,M) \rightarrow 
\Big(\prod_{[H] \in \mathcal{I}} \mathrm{H}^{2}_{\mathcal{TR}[H] }( 
\mathrm{N}_{G}[H],M)\Big)\oplus  \mathrm{H}^{2}(G,M) \rightarrow  
\prod_{[H] \in \mathcal{I}} \mathrm{H}^{2}( \mathrm{N}_{G}[H],M)  \]
\[  \rightarrow \mathrm{H}^{3}_{\mathcal{VC}}(G,M) \rightarrow 
\Big(\prod_{[H] \in \mathcal{I}} \mathrm{H}^{3}_{\mathcal{TR}[H] }( 
\mathrm{N}_{G}[H],M)\Big)\oplus  \mathrm{H}^{3}(G,M) \rightarrow  
\prod_{[H] \in \mathcal{I}} \mathrm{H}^{3}( \mathrm{N}_{G}[H],M).  \]
By Lemma \ref{lem: elliptic}, the set $\mathcal{I}$ contains exactly one element $[K]$ 
such that $K$ is generated by an elliptic element. Also, by the same lemma, we have $\mathrm{N}_G[K]=G$. Moreover, the family 
$\mathcal{TR}[K]$ is the same as the family of all cyclic groups generated by 
elliptic elements of $G$, and $\mathrm{cd}_{\mathcal{TR}[K]}G \leq 1$ 
because the tree $T$ is a model for $E_{\mathcal{TR}[K]}G$. If $H$ is 
generated by a hyperbolic element, then $\mathrm{N}_G[H]\cong 
\mathbb{Z}$ and hence $\mathrm{cd}_{\mathcal{TR}[H]}\mathrm{N}_G[H] =0$. Combining all these facts, the long exact sequence 
reduces to
\begin{equation} \label{eq: long 
exact}\mathrm{H}^{1}_{\mathcal{VC}}(G,M) \rightarrow 
\mathrm{H}^{1}(G,M)\oplus  \mathrm{H}^{1}_{\mathcal{TR}[K] }( G,M) 
\rightarrow   \mathrm{H}^{1}( G,M)\oplus \prod_{[H] \in \mathcal{I}_0} 
\mathrm{H}^{1}( \mathbb{Z},M)  \end{equation}
\[  \rightarrow \mathrm{H}^{2}_{\mathcal{VC}}(G,M) \rightarrow 
\mathrm{H}^{2}(G,M) \xrightarrow{\mathrm{Id}} \mathrm{H}^{2}(G,M) 
\rightarrow \mathrm{H}^{3}_{\mathcal{VC}}(G,M) \rightarrow 0,    \]
where $K$ is generated by an elliptic element of $G$ and $\mathcal{I}_0 
= \mathcal{I} \smallsetminus [K]$. This implies that 
$\mathrm{H}^{3}_{\mathcal{VC}}(G,M)=0$ for every $M$ and hence 
$\underline{\underline{\mathrm{cd}}}(G)\leq 2$.

By Corollary $2$ of \cite{Kropholler}, the second derived subgroup of 
$G$ is a free group $F$. If $F$ is non-abelian, then $G$ contains a 
finitely generated non-abelian free group and hence $ 
\underline{\underline{\mathrm{cd}}}(G)=2$,  by Lemma \ref{lemma: free 
group}. If $F$ is abelian, then $G$ is a solvable group of cohomological 
dimension $2$.  Since $G$ is a finitely generated non-cyclic group that 
does not contain $\mathbb{Z}^2$, we conclude from \cite[Theorem 5]{gilden} that $G$ 
is isomorphic to a solvable Baumslag-Solitar group $BS(1,n)$, with 
$n\neq \pm 1$. It is well-known that, in this case, $G$ can be written 
as a semi-direct product $\mathbb{Z}[\frac{1}{n}] \rtimes \mathbb{Z}$. 
Here the generator $t$ of $\mathbb{Z}$ acts on  
$\mathbb{Z}[\frac{1}{n}]$ as multiplication by $n$. Setting $M$ as the
  trivial module $\underline{\mathbb{Z}}$, we obtain from (\ref{eq: 
long exact}) the exact sequence
\[\mathrm{H}^{1}(G,\mathbb{Z}) \oplus \mathrm{H}^{1}_{\mathcal{TR}[K] }( 
G,\underline{\mathbb{Z}}) \xrightarrow{f}  \mathrm{H}^{1}(G,\mathbb{Z}) 
\oplus \prod_{[H] \in \mathcal{I}_0} \mathrm{H}^{1}( 
\mathbb{Z},\mathbb{Z}) \rightarrow 
\mathrm{H}^{2}_{\mathcal{VC}}(G,\underline{\mathbb{Z}}) . \]
Note that in this exact sequence, $f$ maps $\mathrm{H}^{1}_{\mathcal{TR}[K] }( 
G,\underline{\mathbb{Z}}) $ into 
$\mathrm{H}^{1}(G,\mathbb{Z})\cong \mathbb{Z}$.  Now, if $H$ is an 
infinite cyclic subgroup of $G$ generated by an element $h$ of $G$ that 
is not contained in $\mathbb{Z}[\frac{1}{m}]$, then one verifies 
directly or by Lemma 3.1 of \cite{DP2} that 
$\mathrm{N}_G[H]\cong \Z$. Therefore, by Lemma \ref{lem: elliptic}, $h$ is a hyperbolic element 
of $G$. In particular, the subgroups $H_1=\langle (0,t) \rangle$ and  
$H_2=\langle (1,t) \rangle$ of $G$ are generated by hyperbolic elements. 
Moreover, one can check that $[{}^gH_1]\neq [H_2]$ for all $g \in G$. 
This shows that $\mathcal{I}_0$ contains at least two element. 
Therefore, the map $f$ cannot be surjective and we conclude that 
$\mathrm{H}^{2}_{\mathcal{VC}}(G,\underline{\mathbb{Z}}) \neq 0$. Therefore, 
$\underline{\underline{\mathrm{cd}}}(G)=2$. This finishes the proof.
\end{proof}

\subsection{Graph product of groups} Instead of looking at graphs of groups we can, more generally, consider Haefliger's complexes of groups (see \cite{Haef} and \cite{BridHaef}). However, one needs to be more careful in this context. In contrast to the one-dimensional case, which is the graph of groups, it is not guaranteed that a complex of groups gives rise to an action of its direct limit on a CAT(0)-space. We consider an instance where one does obtain such an action, namely for graph products of groups.

 Let $L$ be a simplicial graph with finite vertex set S and let $\{G_s\}_{s \in S}$ be a collection of countable groups indexed by $S$. The associated graph product $G_L$ is the group generated by the elements of the vertex groups $G_s$, subject to the relations in $G_s$ and the additional relations that $[g_s,g_t]=e$ if $g_s \in G_s$, $g_t \in G_t$ and $s$ and $t$ are joined by an edge in $L$. Let $K$ be the flag complex associated to $L$ and let $\mathcal{Q}$ be the poset of those subsets of $S$ that span a simplex in $K$, including the empty set. Associate to each $\sigma=(s_1,\ldots,s_r) \in \mathcal{Q}$,  the group $G_{\sigma}=\prod_{i=1}^r G_{s_r}$. As explained in \cite[II.12.30(2)]{BridHaef} (see also \cite{Davis} and \cite{Meier}), these data, together with the obvious inclusion maps, form a simple complex of groups $G(\mathcal{Q})$ over the poset $\mathcal{Q}$ with direct limit $\widehat{G(\mathcal{Q})}=G_L$. Moreover,  $G_L$ acts isometrically, cocompactly and cellularly on a complete CAT(0) piecewise Euclidean cubical complex $X_{G_L}$, with stabilizers isomorphic to subgroups of the local groups $G_{\sigma}$, $\sigma \in \mathcal{Q}$. Moreover, $X_{G_L}$ is separable, $G$ acts discretely on $X_{G_L}$because it acts cellularly, and the topological dimension of $X_{G_L}$ is $\dim(K)+1$. By Theorem B, we obtain the following.
 \begin{corollary} \label{cor:  graph product} Let $L$ be a simplicial graph with a finite vertex set $S$ and an associated flag complex of dimension $n$. Let $\{G_s\}_{s \in S}$ be a collection of countable groups indexed by the vertex set $S$ and let $G_L$ be the associated graph product. By considering the action of $G_L$ on $X_{G_L}$, one obtains
\[\underline{\underline{\mathrm{cd}}}(G_L)\leq  \max\{\mms(G_L,X_{G_L}), \mv(G_L,X_{G_L})+1\}+n+1.  \]
\end{corollary}
 
\subsection{Groups acting on Euclidean buildings} Euclidean buildings or rather their geometric realizations provide nice examples of finite dimensional complete CAT(0)-spaces  (see \cite{Davis} or \cite[\S 11.2]{AB}). In fact, these objects have a rigid metric structure  as they are CAT(0)-polyhedral cell complexes with finitely many shapes of cells.  

In what follows, we consider  groups that act on Euclidean buildings, assuming always that such an action is cellular and isometric. 
\begin{corollary} \label{cor: build} Let $G$ be a discrete group that acts on a separable Euclidean building $X_i$  for each $i=1, \dots , r$.  Denote $X=X_1\times \ldots \times X_r$ and  by $n$  the dimension of $X$. Let $G$ act diagonally on $X$. Let $\mathcal{F}$ be a family of subgroups of $G$ such that $X^{H} \neq \emptyset$ for all $H \in \mathcal{F}$. Suppose that there exists an integer $d \geq 0$ such that for each $x \in X$ one has $\mathrm{cd}_{\mathcal{F}\cap G_x}(G_x)\leq d$. Then \[\mathrm{cd}_{\mathcal{F}}(G)\leq n+d \;\;\; \mbox{ and } \;\;\; \underline{\underline{\mathrm{cd}}}(G)\leq \max\{\mms(G, X), \mv(G, X)+1\}+n. \]
\end{corollary} 
\begin{proof} Note that $X$ is a separable CAT(0)-metric space of topological dimension at most $n$. Also, by a result of Bridson (see \cite{Bridson}), $G$ acts by semi-simple isometries on $X$.  So, the assertions follow directly from Theorems A and B.
\end{proof}
As we shall prove next, this result has an immediate application to Bredon cohomological dimension of linear groups of positive characteristic.

\begin{proof}[Proof of Corollary \ref{cor: intro linear}  ]The strategy of the proof is to obtain an action of $G$ as in the statement of  the previous corollary on a finite product of buildings. A construction of Cornick and Kropholler (see \cite[\S 8]{CK}) which is the positive characteristic version of the original construction of  Alperin and Shalen (see \cite{AS}) does exactly this. We briefly outline their argument in order to point out that the product of buildings in this construction is a separable space. 

Since the general linear group $\mathrm{GL}_n(F)$ is isomorphic to a subgroup of $\mathrm{SL}_{n+1}(F)$, we can assume that $G$ is a subgroup of $\mathrm{SL}_n(F)$. Let $S$ be the subring of $F$ generated by the matrix entries of a finite set of generators of $G$ and their inverses. Then $G$ embeds into $\mathrm{SL}_{n}(S)$. So, without loss of generality, we can assume that $G=\mathrm{SL}_{n}(S)$. The ring $S$ is a finitely generated domain and hence it is integral over a polynomial ring $\mathbb F_p[x_1, \dots , x_s]$.  Let $E$ be the fraction field of $S$. It follows that there are finitely many discrete valuation rings $\mathcal O_{v_i}$ of $E$, $1\leq i\leq r$ such that $S\cap \bigcap_{i=1}^r \mathcal O_{v_i}$ is contained in the algebraic closure $L$ of $\mathbb F_p$ in $E$ and $L$ is finite (see \cite[Proposition 8.4]{CK}). 

Let $\hat{E}_i$ denote the completion of $E$ with respect to the valuation $v_i$  and consider  for each $1\leq i\leq r$  the group $\mathrm{SL}_{n}(\hat{E}_i)$.  There is a Euclidean  building $X_i$  of dimension $n-1$  associated  to $\mathrm{SL}_{n}(\hat{E}_i)$ such that this group acts chamber transitively  on $X_i$ and the restriction of the action  to $G$ has vertex stabilizers conjugate to a subgroup of $\mathrm{SL}_{n}(\mathcal O_{v_i})$  (see \cite[\S 6.9, \S 11.8.6]{AB}, \cite[page 61]{CK}). It follows that $G$ acts diagonally on the product $X=X_1\times \ldots \times X_r$ such that each stabilizer subgroup $G_x$ of a vertex $x$ of $X$ lies inside 
$$\mathrm{SL}_{n}(S)\cap \bigcap_{i=1}^r a_i^{-1}\mathrm{SL}_{n}(\mathcal O_{v_i}){a_i}, \;\;  \mbox{ for }  \;\; a_i\in \mathrm{SL}_{n}(E), \; \; i=1, \dots, r.$$ Lemmas 8.6 and  8.7 of \cite{CK} entail that  $G_x$ is  locally finite. For each $1\leq i\leq r$, let $\Sigma_i$ be the fundamental chamber of $X_i$. Since $X_i$ is a continuous image of  the separable space $\mathrm{SL}_{n}(\hat{E}_i)\times \Sigma_i$, it is itself separable. 

Hence, we obtain an action of $G$ on a product $X$ of Euclidean buildings with countable locally finite stabilizers and moreover, $X$ is a separable space. Since countable locally finite groups have Bredon cohomological dimension for the family of finite subgroups at most 1, by Corollary \ref{cor: finite}, we deduce that ${\underline{\mathrm{cd}}}(G)\leq 1 + r(n-1)$. 

Note that $\mv(G, X)\leq 1$ because a finite extension of a locally finite group is again locally finite.  Therefore, Corollary \ref{cor: build} implies $\underline{\underline{\mathrm{cd}}}(G)\leq 2+r(n-1)$.
\end{proof}

\begin{remark}\rm Observe from the proof above that for any finitely generated domain $S$ of positive characteristic $p$ and for a positive integer $n$, one has
$${\underline{\mathrm{cd}}}(\mathrm{SL}_{n}(S))\leq 1+ r(n-1) \;\;\; \mbox{ and } \;\;\; \underline{\underline{\mathrm{cd}}}(\mathrm{SL}_{n}(S))\leq 2+r(n-1)$$
where  $r$ is the minimum number of valuations $v_i$ such that $S\cap \bigcap_{i=1}^r \mathcal O_{v_i}$ is contained in the algebraic closure of $\mathbb F_p$ in the fraction field of $S$.
\end{remark}

\subsection{Mapping class groups}
Let $S_g$ be a closed, connected and oriented surface of genus $g$ and denote by $\mathrm{Homeo}_{+}(S_g)$ the group of orientation preserving homeomorphisms of $S_g$. Equipped with the compact-open topology, $\mathrm{Homeo}_{+}(S_g)$ becomes a topological group. The mapping class group of the surface $S_g$, denoted  $\mathrm{Mod}(S_g)$, is by definition the discrete group
\[\mathrm{Mod}(S_g)= \mathrm{Homeo}_{+}(S_g)/\mathrm{Homeo}^{0}(S_g),     \]
where $\mathrm{Homeo}^{0}(S_g)$  is the identity component of $\mathrm{Homeo}_{+}(S_g)$. Equivalently, one can say that $\mathrm{Mod}(S_g)$ is the group of isotopy classes of orientation preserving homeomorphisms of $S_g$.  Mapping class groups are known to be virtually torsion-free, residually finite and finitely presented. The mapping class group of the $2$-sphere is trivial. The mapping class group of the torus is $\mathrm{SL}(2,\mathbb{Z})$ which is isomorphic to the amalgamated free product $\mathbb{Z}_4 \ast_{\mathbb{Z}_2} \mathbb{Z}_6$. So,  by Lemma \ref{lemma: free group} and Corollary \ref{cor: intro specific stab}, it follows that $\underline{\underline{\mathrm{cd}}}(\mathrm{SL}(2,\mathbb{Z}))= 2$.

From now on, we assume that $g \geq 2$. A marked hyperbolic structure on $S_g$ is a pair $(X,\varphi)$, where $X$ is a closed, connected and orientable surface with a   hyperbolic metric and $\varphi: S_g \rightarrow X$ is a diffeomorphism.  The map $\varphi$ is called a marking. Two marked hyperbolic structures $(X,\varphi)$ and $(X',\varphi')$ on $S_g$ are called equivalent if there exists an isometry $\theta: X \rightarrow X'$ such that $\theta \circ \varphi$ is homotopic to $\varphi'$. The Teichm\"{u}ller space $\mathcal{T}(S_g)$ of $S_g$ is by definition the space of equivalence classes of marked hyperbolic structures on $S_g$. The space $\mathcal{T}(S_g)$ can be given a natural topology under which it is homeomorphic to $\mathbb{R}^{6g-6}$. It is known that the mapping class group $\mathrm{Mod}(S_g)$ acts properly on $\mathcal{T}(S_g)$  by precomposing markings with diffeomorphisms (every self-homeomorphism of $S_g$ is isotopic  to a diffeomorphism (see \cite[1.13]{Farb})). Moreover, $\mathcal{T}(S_g)$ has an equivariant triangulation making it  a model for $\underline{E}\mathrm{Mod}(S_g)$ (see \cite{mislin} and the references therein). We refer the reader to \cite{Farb} for more details about stated and other results on  mapping class groups and Teichm\"{u}ller space.

There are several metrics that can be given to $\mathcal{T}(S_g)$ that all give rise to the same topology. However, the metric properties of $\mathcal{T}(S_g)$ under these different metrics vary considerably.  The metric we will be interested in is the so-called Weil-Petersson metric $d_{WP}$. Equipped with the Weil-Petersson metric, the  Teichm\"{u}ller space $\mathcal{T}(S_g)$ is a non-complete separable CAT(0)-space of dimension $6g-6$, on which $\mathrm{Mod}(S_g)$ acts by isometries. The completion of $\mathcal{T}(S_g)$ with respect to the Weil-Petersson metric is the augmented Teichm\"{u}ller space $\overline{\mathcal{T}}(S_g)$. Roughly stated, the augmented Teichm\"{u}ller space is a stratified space whose strata are Teichm\"{u}ller spaces associated to nodal surfaces obtained by shrinking essential simple closed loops on $S_g$ to pairs of cusps. It follows that $(\overline{\mathcal{T}}(S_g),d_{WP})$ is a complete separable CAT(0)-space of dimension $6g-6$ on which $\mathrm{Mod}(S_g)$ acts by isometries. Although $\overline{\mathcal{T}}(S_g)$ is not locally compact, the quotient space of $\overline{\mathcal{T}}(S_g)$ obtained by modding out the action of $\mathrm{Mod}(S_g)$ is compact. It is the so-called Deligne-Mumford compactification of the moduli space of curves. We refer to the survey papers \cite{wolpert} and \cite{wolpert2} for the definition of the Weil-Petersson metric and the augmented  Teichm\"{u}ller space, and for the references concerning the properties stated above. 

\begin{proof}[Proof of Corollary \ref{cor: intro mapping class group}]
As we already mentioned, the mapping class group $\mathrm{Mod}(S_g)$ acts by isometries on the augmented Teichm\"{u}ller space $\overline{\mathcal{T}}(S_g)$, which is a complete separable CAT(0)-space of dimension $6g-6$. Moreover, according to Theorem $A$ in \cite{Bridson2} the action is by semi-simple isometries. We claim that this action  is, in addition, discrete and that the isotropy groups are finitely generated virtually abelian of Hirsch length at most $3g-3$. Assuming these claims, the result follows from Corollary \ref{cor: intro specific stab}(iii). 

Now, let us prove our claims.  Since $\mathrm{Mod}(S_g)$ acts properly on $\mathcal{T}(S_g)$, all orbits of points in  $\mathcal{T}(S_g)$ are discrete and the stablizers of all points in $\mathcal{T}(S_g)$ are finite. It remains to consider points in $\overline{\mathcal{T}}(S_g)\smallsetminus \mathcal{T}(S_g)$.  This space is a union of strata $\mathcal{S}_{\Gamma}$ corresponding to sets $\Gamma$ of free homotopy classes of disjoint essential simple closed curves on $S_g$. Let $x \in \mathcal{S}_{\Gamma} $ and let $\Delta_{\Gamma}$ be the group generated by the Dehn twists defined by the curves in $\Gamma$. The group $\Delta_{\Gamma}$ is free  abelian of rank at most $3g-3$ and fixes $\mathcal{S}_{\Gamma}$ pointwise. Hence, $\Delta_{\Gamma}$ is contained in the point stabilizers of $x$. Moreover, Corollary 2.7 in \cite{hubbard} states that there exists an open neighbourhood $U \subseteq \overline{\mathcal{T}}(S_g)$ of $x$ such that the set
\[ \{  g \in \mathrm{Mod}(S_g) \ | \ g\cdot U \cap U \neq \emptyset            \}\]
is a finite union of cosets of $\Delta_{\Gamma}$. This implies that the point stablizer of $x$ contains $\Delta_{\Gamma}$ as a finite index subgroup. Also, it is not difficult to see that we can now choose an open ball $\mathrm{B}(x,\varepsilon) \subseteq U$ such that
\[ g \cdot \mathrm{B}(x,\varepsilon)\cap \mathrm{B}(x,\varepsilon) \neq \emptyset \Leftrightarrow g \in \mathrm{Mod}(S_g)_{x}. \]
Hence, the orbit of $x$ is discrete. This completes the proof.
\end{proof}
\begin{center}\textbf{Acknowledgements}\end{center}
The authors are grateful to Ralf K\"{o}hl for a helpful correspondence on Bruhat-Tits buildings and for recalling the result of Bridson on the semisimplicity of polyhedral isometries. They also thank Richard Wade for valuable discussions about mapping class groups and their actions on Teichm\"{u}ller spaces.

\end{document}